\documentclass{amsart}
\usepackage{amssymb,amsmath,amsthm,amscd}

\usepackage{color}

\usepackage{ulem} 
\usepackage{enumerate}
\newcommand{\annotation}[1]{\marginpar{\tiny #1}}

\newcommand{\blue}[1]{\textcolor{blue}{#1}}

\newcommand{\Z}{\mathbb Z}
\newcommand{\Q}{\mathbb Q}
\newcommand{\R}{\mathbb R}

\newcommand{\N}{\mathbb N}
\newcommand{\cok}{\text{cok}}

\newcommand{\mapt}[1]{\mathsf{S}{#1}}

\numberwithin{equation}{section}
\newtheorem{theorem}[equation]{Theorem}
\newtheorem{corollary}[equation]{Corollary}

\newtheorem{lemma}[equation]{Lemma}
\newtheorem{proposition}[equation]{Proposition}

\theoremstyle{definition}
\newtheorem{definition}[equation]{Definition}
\newtheorem{example}[equation]{Example}

\newtheorem{question}[equation]{Question}

\theoremstyle{remark}
\newtheorem{remark}[equation]{Remark}

\begin{document}

\title{The mapping class group of a shift of finite type}


\author{Mike Boyle and Sompong Chuysurichay}

\address{
Department of Mathematics, University of Maryland, College
Park, MD 20742-4015, USA
}
\email{mmb@math.umd.edu}

\address{Algebra and Applications Research Unit, Department of Mathematics and Statistics, Prince of
Songkla University, Songkhla, Thailand 90110}
\email{sompong.c@psu.ac.th}

\subjclass[2010]{Primary 37B10; Secondary 20F10, 20F38.}






\keywords{shift of finite type, flow equivalence, mapping class group,
automorphism group}

\dedicatory{Dedicated  to Roy Adler, in memory of his
  insight, humor and kindness}

\begin{abstract}
  Let $(X_A,\sigma_A)$ be a nontrivial irreducible
  shift of finite type (SFT), with
  $\mathcal{M}_A$ denoting its
  mapping class group:  the group of
flow equivalences of its mapping torus $\mapt X_A$, (i.e., self
homeomorphisms of
$\mapt X_A$
 which respect the direction of the
suspension flow) modulo the subgroup of flow equivalences of
$\mapt X_A$
isotopic to the identity.
We develop
and apply machinery (flow codes,
  cohomology constraints) and provide
 context
for the study of $\mathcal M_A$, and prove  results including
the following.
 $\mathcal{M}_A$ acts faithfully and
 $n$-transitively (for every $n$ in $\mathbb{N}$)
 by permutations on the set
of circles of $\mapt X_A$. The
center of $\mathcal{M}_A$ is trivial.
The outer automorphism group of $\mathcal{M}_A$ is nontrivial.
In many cases, $\text{Aut}(\sigma_A)$ admits a nonspatial automorphism. 
For every SFT
$(X_B,\sigma_B)$ flow equivalent to $(X_A,\sigma_A)$,
$\mathcal{M}_A$ contains
embedded copies of $\text{Aut}(\sigma_B)/\left<\sigma_B\right>$,
induced by return maps to invariant cross sections; but,
elements of $\mathcal M_A$
not arising from
flow equivalences with invariant cross sections are abundant.
$\mathcal{M}_A$ is countable and
has solvable word problem.
$\mathcal{M}_A$ is not residually finite.
Conjugacy classes of many (possibly all) involutions
in $\mathcal M_A$ can be classified by the $G$-flow equivalence
classes of associated $G$-SFTs, for $G=\Z/2\Z$.
There are many open questions.
\end{abstract}

\maketitle




\tableofcontents

\section{Introduction}

Throughout this paper,  $T: X\to X$ denotes
 a homeomorphism of a compact zero dimensional metric space $X$,
and $\mapt (X,T)$
 is the mapping torus of $T$, which carries
a natural suspension flow (detailed definitions are in
Sec.  \ref{sec:definitions}).  We usually write $\mapt X$ for
$\mapt (X,T)$.
Let $\mathcal F(T)$ denote the group of self equivalences of
the suspension flow on $\mapt X$, i.e.,  the  homeomorphisms
$\mapt X \to \mapt X$ which map orbits
onto orbits, respecting the direction of the flow.
Define the mapping class group of $T$,
 $\mathcal M(T)$, to be the group of isotopy classes of
elements of $\mathcal F(T)$.  By definition,
for $h$ in $\mathcal F(T)$, the class $[h]$ is
trivial in   $\mathcal M(T)$ if there is a continuous
map $\mapt X \times [0,1] \to \mapt X$, $(y,t) \mapsto h_t(y)$,  with $h_0$ the identity,
$h_1=h$ and each  $h_t$ in $\mathcal F(T)$.
Because  $X$ is zero dimensional,
this condition forces each $h_t$ to map each flow orbit to itself.
The automorphism group of $T$, $\text{Aut}(T)$, is the
group of homeomorphisms $X\to X$ which commute with $T$.

For an irreducible matrix $A$ over $\Z_+$, let
$\sigma_A: X_A\to X_A$
be the associated shift of finite type (SFT).
We say an SFT is  trivial if $X_A$ is a single finite orbit.
Let
$\mathcal M_A=\mathcal M(\sigma_A)$.
In this paper we study  $\mathcal M_A$,
the mapping class group of an
irreducible shift of finite type, introduced in \cite{Bo3}. (Several of
the results, along with  ingredients of some others,
 appeared in the Ph.D. thesis of S. Chuysurichay
\cite{Cthesis2011}.)

Homeomorphisms $T,T'$ are {\it  flow equivalent}  if the suspension
flows on their mapping tori are equivalent, i.e. there is a
homeomorphism $h: \mapt X\to \mapt X'$ mapping orbits onto orbits,
respecting the orientation of the flow.
Here, $h$ induces an isomorphism
$\mathcal M(T) \to \mathcal M(T')$.
$\mathcal M(T)$ plays 
for flow equivalence the role that $\text{Aut}(T)$ plays for
topological conjugacy.
Flow equivalence is very naturally a part of unified algebraic
framework for classifying SFTs (see e.g. \cite{Bo1}). 
A classification of SFTs up to flow equivalence is
known; the classification, and some of the ideas involved,
have been
quite useful for the stable and unital classification of Cuntz-Krieger
algebras
(e.g. \cite{Restorff,Rordam})
 and  more generally, graph $C^*$-algebras
(e.g. \cite{ERRS}).
The track record of utility for flow equivalence is
another motivation
for looking at  $\mathcal M_A$.

We will see that for a nontrivial irreducible
SFT $\sigma_A$, $\mathcal M_A$ contains naturally embedded copies of
$\text{Aut} (\sigma_B)/\left<\sigma_B\right>$, for every $\sigma_B$
flow equivalent to $\sigma_A$, where  $\left<\sigma_B\right>$ is
the subgroup consisting of the powers of $\sigma_B$.
Automorphism groups of SFTs are still poorly understood,
despite longstanding interest
(e.g. \cite{H,  BLR, KRW92});
this relation to automorphism groups  is another reason for our interest
in $\mathcal M_A$, particularly given a resurgence of interest
in automorphism groups of various symbolic systems
(e.g. \cite{CQY, CFKP, CK2016, DDMP2016, Hochman2010,
Schraudner2006,ST2015}.)

We are also interested in
 $\mathcal M_A$ as a large (though countable)
dynamically defined group.
Some such groups arising from zero dimensional  dynamics have
turned out to be quite interesting as countable groups
(e.g. \cite{GM,JM,Matui}.)
And although
the groups $\mathcal M_A$ are quite different from the mapping class
groups of surfaces,  it is not impossible that from the vast wealth
of ideas and tools in the surface case (see \cite{FM2012}) 
some  useful approach to $\mathcal M_A$ may be suggested.

We turn now to the organization of the paper. 
In Section \ref{sec:definitions}, we give background.
For a nontrivial irreducible SFT $\sigma_A$,
the action of $\text{Aut}(\sigma_A)$ on
finite invariant sets of periodic points
has been a key tool for progress (e.g. in \cite{KRW92}).
In Section \ref{sec:actions}, we show nothing like this
is available to study $\mathcal M_A$:
 for every $n\in\mathbb{N}$,
$\mathcal{M}_A$ acts $n$-transitively and faithfully
on the set of circles in $\mapt X_A$.
The  other general tool
which has proved  useful for studying
$\text{Aut}(\sigma_A)$
(especially with respect to its action on periodic points
\cite{KRW92}, via Wagoner's Strong Shift Equivalence spaces \cite{WaBAMS})
is the dimension representation, $\rho_A$. 
The
analogue of $\rho_A$ for $\mathcal M_A$ is the Bowen-Franks representation,
$\beta_A$, which for a nontrivial irreducible SFT $\sigma_A$ maps
$\mathcal M_A$ onto the group of group automorphisms of
the Bowen-Franks group $\text{coker}(I-A)$ \cite{Bo3}.
Among our
questions: is the kernel of $\beta_A$ simple?
finitely generated?  sofic?

In Section \ref{sec:actions}, we also show the actions of
$\mathcal M_A$ on circles of $\mapt X$ (by permutations) and
on $\check{H}^1(\mapt X)$ are faithful, and
prove
an analogue of Ryan's Theorem for 
$\text{Aut}(\sigma_A)$:
the center of $\mathcal{M}_A$ is trivial.

In Section \ref{sec:autos}, we  show $\mathcal M_A$ has a
nontrivial outer automorphism group, and
(extending work of \cite{mbwk:amsess})  
for  many  mixing SFTs $\sigma_A$ construct a
group isomorphism $\text{Aut}(\sigma_A)\to \text{Aut}(\sigma_A)$
which is not spatial: i.e., is not induced by a homeomorphism.
We also show that spatial  isomorphism of sufficiently
rich subgroups is enough to imply flip conjugacy.

In Section \ref{sec:sections}, we describe how
flow equivalences $\mapt X\to \mapt X$ with invariant cross sections are the
flow equivalences induced by automorphisms of maps $S$
flow equivalent to $T$,
and
show that by this correspondence
$\mathcal{M}_A$ contains embedded copies of
$\text{Aut}(\sigma_B)/\left<\sigma_B\right>$ for any SFT $(X_B,\sigma_B)$
flow equivalent to $(X_A,\sigma_A)$.
Appealing to a general extension result from \cite{BCEbce},
we also show that 
for any
nontrivial irreducible
SFT $(X_A, \sigma_A)$,
 there is an abundant supply of elements in $\mathcal M_A$
containing no
flow equivalence with an invariant cross section.
We also give a  concrete example of such an element, 
 not appealing to an extension theorem,

In Section \ref{sec:residual},
we show that $\mathcal{M}_A$ is not residually finite.
In Section \ref{sec:solvable}, we show that 
$\mathcal{M}_A$ has  solvable word problem.
In Section \ref{Conjugacyclassesofinvolutions}, 
we give results on conjugacy classes of involutions
in $\mathcal M_A$ by establishing a connection to
the theory of $\Z_2$-SFTs. For example,
if $\det(I-A)$ is odd, then only finitely many
conjugacy classes in $\mathcal M_A$ can contain fixed
point free involutions. 

At points in the paper we make use of flow codes,
a flow
analogue of block codes, introduced in
\cite{BCEfei}. For Section \ref{sec:solvable},
we also need to address composition of
flow codes up to isotopy. The background and new work
on flow codes is given in  Appendix \ref{sec:flowcodes}.

In the course of the paper we make explicit several of the many
open questions about $\mathcal M_A$.

\section{Definitions and background}\label{sec:definitions}


There is more detailed background on the material below in
  \cite{LM}  (for Sec. 2.1),
\cite{BCEfei} (for Secs. 2.2 and 2.3) and \cite{Bo3} (for
Secs. 2.4 and 2.5).

\subsection{Shifts of Finite Type}

Let $A$ be an $n\times n$ nonnegative integral matrix. $A$ can be
viewed as an adjacency matrix of a finite directed graph $G$ with
$n$ ordered vertices and a finite edge set $E$ and $A_{ij}$ is the
number of edges from vertex $i$ to vertex $j$. Let $X_A$ be the
subspace of $E^{\mathbb{Z}}$ consisting of bi-infinite sequences
$(x_i)$ such that for all $i\in\mathbb{Z}$, the terminal vertex of
$x_i$ is the initial vertex of $x_{i+1}$. Then with the subspace
topology from the product topology of $E^{\mathbb{Z}}$, $X_A$ is a
compact metrizable space and the \textit{shift map} $\sigma_A$
defined by the rule $(\sigma_A(x))_i=x_{i+1}$ is a homeomorphism
from $X_A$ to $X_A$. $(X_A,\sigma_A)$ is a \textit{shift of finite
type (SFT) defined by $A$}.  In general an SFT is any dynamical
system topologically conjugate to some $(X_A,\sigma_A)$;
  in addition, $A$ can be chosen nondegenerate (no zero row or column).
An SFT
$(X_A,\sigma_A)$ is \textit{irreducible} if it has a dense forward
orbit; it is \textit{trivial} if $X_A$ is a finite set.  For $A$ nondegenerate,  $(X_A,\sigma_A)$ is irreducible if and only
if $A$ is an irreducible matrix. If $A$ is irreducible, then
$(X_A,\sigma_A)$ is trivial if and only if $A$ is a
cyclic permutation
matrix.

\subsection{Suspensions, Cross Sections, and Flow Equivalences}
For a homeomorphism
$T:X\to X$, we define its  \textit{mapping
torus} $\mapt (X,T)=\mapt X$ to be the quotient space
$ (X \times \R)/ \sim\ $,
where $(x,t)\sim (T^n(x),t-n)$ for $n\in \Z$ and $t\in \R$.
We write the image of $(x,t)$ in $\mapt X$ as $[x,t]$.
An element of $\mapt X$
may be
represented as  $[x,t]$ for a unique $x$ in $X$ and $t$ in $[0,1)$.
   For any
$s\in\mathbb{R}$, the \textit{suspension flow}
$\alpha : \mapt X \times \R \to \mapt X $
is defined by $([x,t],s) \mapsto \alpha_s([x,t])=[x,s+t]$.
Two discrete dynamical systems
$(X,T)$ and $(X',T')$ are \textit{flow equivalent} if there is a
homeomorphism $F:\mapt X\to \mapt X'$  mapping flow orbits onto
flow orbits, respecting the direction of the flow.
$F$ is called a \textit{flow equivalence}.
Any conjugacy of
discrete dynamical systems induces a
topological conjugacy of the corresponding suspension flows
(and this is a
flow equivalence), but in general flow equivalence is a much
weaker equivalence relation.

A \textit{cross section} $C$ of the suspension flow $\alpha$ on
$\mapt X$ is a closed set of $\mapt X$ such that $\alpha
:C\times\mathbb{R}\to \mapt X$ is a local homeomorphism onto $\mapt X$
\cite{Schwartzman57}. It
follows that every orbit hits $C$ in forward time and in backward
time, \textit{the first return time} defined by
$f_c(x)=\inf\{s>0:\alpha_s(x)\in C\}$ is continuous and strictly
positive on $C$, and \textit{the first return map} $\rho_c:C\to C$
defined by $\rho_c(x)=\alpha_{f_c(x)}(x)$ is a homeomorphism.
Discrete systems $(X,T)$ and $(X',T')$ are flow equivalent if and
only if there is a flow $Y$ with two cross sections whose return
maps are conjugate respectively to $T$ and $T'$.

We define the \textit{mapping
class group} of $T$, denoted by $\mathcal{M}(T)$, to be the group of
flow equivalences $\mapt X\to \mapt X$ modulo the subgroup of flow
equivalences which are isotopic to the identity
in $\mathcal F(T)$.
Two
flow equivalences $F_0,F_1:\mapt X \to \mapt X'$ are \textit{isotopic} if
$[(F_1)^{-1}F_0]$ is trivial in $\mathcal M(T)$.
By definition,
for $h$ in $\mathcal F(T)$, the class $[h]$ is
trivial in   $\mathcal M(T)$ if there is a continuous
map $\mapt X \times [0,1] \to \mapt X$, $(y,t) \mapsto h_t(y)$,  with $h_0$ the identity,
$h_1=h$ and each  $h_t$ in $\mathcal F(T)$.
When $T$ is a shift of finite type $\sigma_A$, we may write
$\mathcal F_A$ for $\mathcal F(T)$ and
$\mathcal M_A$ for $\mathcal M (T)$.

\begin{question} \label{splitquestion}
  Does the epimorphism $\mathcal F(\mapt X_A) \to \mathcal M_A$ given
  by $F \mapsto [F]$
split?
\end{question}

Question \ref{splitquestion} asks whether
$\mathcal M_A$ can  be presented as a group of homeomorphisms.

\subsection{The Parry-Sullivan Argument}

 A {\it discrete cross section} for
a homeomorphism $T: X\to X$ is a closed subset $C$  of
$X$ with a continuous function
$r \colon  C \to  \N$ such that
$r(x) = \min \{k \in \mathbb{N} : T^k(x) \in C\} \text{ and } X = \{T^k(x) : x
\in C,k \in \mathbb{N} \}$.
When $X$ is zero dimensional,
the set $C$ must be  clopen in $X$,
by continuity of the return time function
$r$.

The argument of Parry and Sullivan in \cite{PS} shows the following.

\begin{theorem}\label{pstheorem}  (\cite{PS}; see
\cite[Theorem 4.1]{BCEfei}) Suppose
$Y,Y'$ are one dimensional compact metric spaces
with fixed point free flows $\gamma, \gamma'$
for which $C,C'$
are zero dimensional cross sections, with return maps
 $\rho_C, \rho_{C'}$.
Suppose
$h: Y \to Y'$ is a flow equivalence.

Then there are
discrete cross sections $D,D'$ for $\rho_C, \rho_{C'}$,
with 
  $D\subset C$ and $D' \subset C'$, 
such that $h^{-1}(D')=F(D)$ for some $F$ isotopic to the
  identity,  and    
$h$ is isotopic to a homeomorphism $Y\to Y'$ induced by a
topological conjugacy
$(D,\rho_D ) \to
(D',\rho_{D'})$.
\end{theorem}

Theorem \ref{pstheorem} is implicit in the succinct paper \cite{PS};
see \cite{BCEfei}  for full details, generalization and related
examples.

 As a consequence of Theorem \ref{pstheorem}, we have
the following fact.

\begin{corollary}\label{theorem:224}
The mapping class group of a subshift $(X,\sigma)$
  is countable.
\begin{proof}
Let $Y$ be the mapping torus of $X$.
For any discrete cross section $D$ for $(X,\sigma )$, the
system $(X, \rho_D )$ is expansive and therefore topologically
conjugate to a subshift. By Theorem \ref{pstheorem}, up to
isotopy a flow  equivalence $Y \to Y$ is determined by the
choice of clopen sets $D,D'$ and a topological conjugacy
$(D,\rho_D ) \to
(D',\rho_{D'})$ (which can be defined by a block code). There are only
countably many clopen sets in $D$ and only countably  many block
codes.
%
Therefore the mapping class group of $(X,\sigma) $ is countable.
\end{proof}
\end{corollary}

For a simple example in contrast to Corollary \ref{theorem:224},
note that $\mathcal M(T)$
is uncountable if
 $T$ is the identity map on a Cantor set.


\subsection{Positive Equivalence}

Let $A$ and $B$ be irreducible matrices. We embed $A$ and $B$ to
the set of \textit{essentially irreducible} infinite matrices over
$\mathbb{Z}_+$, those which have only one irreducible component.
Within the \lq\lq
positive K-Theory\rq\rq approach to symbolic dynamics
\cite{Bo1,BoW04,Wa00},
there is the general \lq \lq positive equivalence\rq\rq\
method for constructing flow equivalences for SFTs
(developed in
\cite{Bo1}, building on Franks' work \cite{F1}).
({\it Flow codes,} a flow equivalence analogue of
  block codes developed in
\cite{BCEfei}, give a  general presentation of flow equivalences
up to isotopy
for subshifts.) A \textit{basic elementary matrix} $E$ is a matrix in
$\text{SL}(\mathbb{Z})$ which has off-diagonal entry $E_{ij}=1$
where $i\neq j$ and $1$ on the main diagonal and $0$ elsewhere. We
define four \textit{basic positive equivalences} as follows:
suppose $A_{ij}>0,$
\begin{align*}
(E,I) &: I-A\to E(I-A), &(E^{-1},I) : E(I-A)\to I-A \\
(I,E) &: I-A\to (I-A)E, &(I,E^{-1}) : (I-A)E\to I-A.
\end{align*}
A \textit{positive equivalence} is the composition of basic
positive equivalences $(E_i,F_i)$,
$(U,V) = (E_k\cdots E_1,F_1\cdots F_k)$.
 We will only discuss the flow equivalence
induced by the basic positive equivalence $(E,I):I-A\to E(I-A)$.
We can apply the same idea with the others. Define $A'$ from the
equation $E(I-A)=I-A'$. Then $A$ and $A'$ agree except in row $i$,
where we have
\begin{align*}
A'_{ik} & = A_{ik}+A_{jk}\,\,\text{if}\,\, j\neq k, \,\,\text{and} \\
A'_{ij} & = A_{ij} + A_{jj} - 1.
\end{align*}
Let $\mathcal{G}_A$ be a directed graph having $A$ as the
adjacency matrix with edge set $\mathcal{E}_A$. We can describe a
directed graph $\mathcal{G}_{A'}$ which has $A'$ as its adjacency
matrix as follows. Pick an edge $e$ which runs from a vertex $i$
to a vertex $j$ in $\mathcal{G}_A$($e$ exists because $A_{ij}>0$
by assumption). The edge set $\mathcal{E}_{A'}$ will be obtained
from $\mathcal{E}_A$ as follows:

a) remove $e$ from $\mathcal{E}_A$.

b) For each vertex $k$, for every edge $f$ in $\mathcal{E}_A$
from $j$ to $k$ add a new edge named $[ef]$ from $i$ to $k$.

Let $\mathcal{E}_A^*$ be the set of new edges obtained from the
above construction. Define a map $\gamma: \mathcal{E}_{A'}\to
\mathcal{E}_A^*$ by $\gamma (f)=f$ and $\gamma([ef])=ef$. Then
$\gamma$ induces a map $\widehat{\gamma}: X_A\to X_{A'}$, defined
by the rule $$\widehat{\gamma}: \cdots
x'_{-2}x'_{-1}.x'_0x'_1\cdots\mapsto
\cdots\gamma(x'_{-2})\gamma(x'_{-1}).\gamma(x'_0)\gamma(x'_1)\cdots$$
Define $F_{\gamma} : \mapt X_{A'}\to \mapt X_A$ by
setting, for $x$ in $X_{A'}$ and $0\leq t < 1$,
$$F_{\gamma}([x,t]) =
\begin{cases}
[\widehat{\gamma}(x),2t], & \mbox{if }x\in X_{[ef]}\mbox{ for every
  edge of the form $[ef]$} \\
[\widehat{\gamma} (x),t],  & \mbox{otherwise .}
\end{cases}$$
where for an edge $d$ in
$\mathcal E_{A'}$, $X_d =\{x\in X_A: x_0= d \}$.

Then $F_{\gamma}$ is a flow equivalence
(in particular, surjective, even though
$\widehat{\gamma}$ is not).

\subsection{The Bowen-Franks representation}

The
\textit{Bowen-Franks group}
of an $n\times n$ integral matrix $A$ is
$\text{coker}(I-A)=\mathbb{Z}^n/(I-A)\mathbb{Z}^n$. For a shift of
finite type $(X_A,\sigma_A)$, Parry and Sullivan \cite{PS}
showed $\det(I-A)$
is an invariant of flow equivalence,
Bowen and Franks \cite{BowF1}  showed
$\text{coker}(I-A)$ is an invariant of
flow equivalence, and Franks \cite{F1} showed
these invariants are complete
for nontrivial irreducible
shifts of finite type.
There is a complete classfication of general SFTs
up to flow equivalence,
due to Huang \cite{Bo3,BoHuang}, but the general
invariant is much more complicated.

Let $(X_A,\sigma_A)$ be a nontrivial irreducible shift of finite type.
Let $(U,V):(I-A) \to
U(I - A)V=I-A$  be a positive equivalence and
let $ F_{(U,V) }$ be an associated flow equivalence.
(There can be many factorizations of $(U,V)$ into
basic positive equivalences, and they can  define
isotopically distinct flow equivalences.)
We
define $F^*_{(U,V)}: \text{coker}(I -A) \to  \text{coker}(I -A)$
 by the rule $[u] \to [uV]$  (we use the $(U,V )$
action on row vectors to define $\text{coker}(I - A)$.) Then $F^*_{(U,V)}$
is an isomorphism.
Let $\text{Aut}(\text{coker}(I-A))$
denote the group of group automorphisms of $\text{coker}(I - A)$. We define the map
$\rho_A : \mathcal M_A \to  \text{Aut}(\text{coker}(I -A))$
 by the rule $\rho_A : F_{(U,V)} \to F^*_{(U,V)}$ and call $\rho_A$  the {\it
   Bowen-Franks representation} (in \cite{Bo3}, this is called the isotopy
futures representation).  It was proved in \cite{Bo3}  that this rule  gives
a well defined group epimorphism. In
contrast, it was proved in \cite{KRW92} that there can be automorphisms of the
dimension module of
$( X_A,\sigma_A)$ (as an ordered module)   which are not induced by any element of
$\text{Aut}(X_A)$.

\section{Actions, representations and group isomorphisms} \label{sec:actions}

%


The following result is fundamental for studying
the mapping class group of an irreducible SFT.
\begin{theorem}\label{theorem:231}
Suppose $(X_A,\sigma_A)$ is an irreducible SFT and
$F\in\mathcal{F}_A$. Then the following are equivalent.
\begin{enumerate}
\item $F$ is isotopic to the identity.

\item $F(\mathcal{O})=\mathcal{O}$ for all suspension flow orbits $\mathcal{O}$ in $\mapt X_A$.

\item $F(\mathcal{C})=\mathcal{C}$ for all circles

\item $F(\mathcal{C})=\mathcal{C}$ for all but finitely many circles
$\mathcal{C}$ in $\mapt X_A$.
\end{enumerate}
\begin{proof}
The implications $(1) \implies (2) \implies (3)\implies (4) $
hold generally, i.e. with $(X,T)$ in place of  $(X_A, \sigma_A)$, for $T$ a
zero dimensional compact metric space $X$. In the case that
$(X,T)$ is an irreducible SFT, the implication $(2)\implies (1)$ is
\cite[Theorem 6.2]{BCEfei}. Given (3), it follows from
\cite[Theorem 6.1]{BCEfei} that the flow equivalence $F$ up to isotopy
is
induced by an automorphism of the irreducible SFT. As recalled
in the proof of  \cite[Theorem 6.2]{BCEfei}, an
automorphism of an irreducible SFT which fixes all (or even
all but finitely many) orbits must be a power of the shift
\cite[Theorem 2.5]{BK2}.

It remains to show $(4)\implies (3)$. Suppose $U$
is a word such that $\dots UUU \dots$ 
represents a periodic orbit of the irreducible SFT $\sigma_A$
such that for the corresponding circle
$\mathcal C(U)$ 
in $\mapt X_A$, $F(\mathcal C(U))\neq \mathcal C(U)$.  Then one can construct
a word $W$ 
such that for all positive integers $n$,  the words
$WU^n$ represent distinct periodic orbits, with
$F(\mathcal C(WU^n)) \neq \mathcal C(WU^n)$.
So, if $F$ moves one circle outside itself, then $F$ moves infinitely
many circles to different circles, and therefore $(4) \implies (3)$.
\end{proof}
\end{theorem}

\begin{remark} The implication $(2)\implies (1)$ of Theorem \ref{theorem:231}
fails in general, even for some reducible SFTs and mixing sofic shifts
(see \cite[Example 6.1]{BCEfei}).
\end{remark}

Suppose $T: X\to X$ is a
homeomorphism
of a compact zero dimensional metric space.
Then $T$ acts on $C(X,\Z )$, the group of continuous
  functions from $X$ to the $\Z$, by the rule $f\mapsto f\circ T$. 
The following groups are
isomorphic:
the first \v{C}ech cohomology group  $\check{H}^1(\mapt X)$;
the group $C(X,\mathbb{Z})/(I-T)C(X,\mathbb{Z})$;
 the Bruschlinsky group $C(\mapt X, S^1)/\sim $ of continuous maps from $\mapt X$ to the circle modulo
isotopy. (For some exposition, see \cite{BH1}.)
 The group $C(X,\mathbb{Z})/(I-T)C(X,\mathbb{Z})$
is of considerable interest  for dynamics (see \cite{BH1,GPS1995,KRW},
their references and their citers). A flow equivalence $F: \mapt X\to \mapt X$
induces an automorphism of each of these groups; for example,
the automorphism of
$C(\mapt X, S^1)/\sim $ is defined by the obvious rule
$[f]\mapsto [f\circ F]$.

\begin{corollary}\label{theorem:232}
Suppose $(X_A,\sigma_A)$ is a nontrivial irreducible shift of finite
type. Then the action (by permutations) of
the mapping class group $\mathcal{M}_A$ on
the set of circles of $\mapt X_A$ is faithful. The action of $\mathcal{M}_A$ on
the first \v{C}ech cohomology group of $\mapt X_A$ is  also faithful.
\begin{proof}
This follows from Theorem~\ref{theorem:231}, since a homeomorphism
moving a circle in $\mapt X$ to another circle has nontrivial action on
\v{C}ech cohomology.
\end{proof}
\end{corollary}

An important fact for  analyzing the automorphism group of an
irreducible SFT, and its actions, is that there are finite invariant
sets (points of some period), whose union is dense. The next
result (from \cite{Cthesis2011}) shows in a strong way that we have nothing like that for
the study of $\mathcal M_A$.

\begin{theorem} \label{theorem:234}
Suppose $(X_A,\sigma_A)$ is a nontrivial irreducible shift of finite
type. Then
$\mathcal{M}_A$ acts $n$-transitively on the set of circles in
$\mapt X_A$ for all $n\in\mathbb{N}$.
\end{theorem}
\begin{proof}
Let $\{\mathcal{C}_1,\ldots,\mathcal{C}_n\}$ and
$\{\mathcal{C}'_1,\ldots,\mathcal{C}'_n\}$ be sets of $n$ distinct
circles. For each $i\in\{1,2,\ldots,n\}$, let $x_i,x'_i$ be
representatives of the circles $\mathcal{C}_i,\mathcal{C}'_i$
respectively. We take a $k$-block presentation of $(X_A,\sigma_A)$
where $k$ is large enough that any point of period $p$ comes from
a path of length $p$ without repeated vertices except initial and
terminal vertices and no two of these loops share a vertex. If one
of these loops, say $L$, has length greater than $1$, then we
apply a basic positive equivalence which corresponds to cutting
out an edge $e$ on the loop $L$ and replacing it with edges
labeled $[ef]$, for the edge $f$ following $e$. The new loop will
have length $p-1$ in the new graph. Continuing in the same
fashion, we get a loop of length $1$. Since no two of these loops
share a vertex, we can apply the same idea to another loop without
changing the former loop. Continuing in this way, we get a graph
with loops $y_1,\ldots,y_n,y'_1,\ldots,y'_n$ of length $1$, each of
which comes from the loop containing $x_1,\ldots,x_n,x'_1,\ldots,x'_n$.
If necessary we continue to apply basic positive equivalences
until we get a graph $\mathcal{G}_B$  with at least one point of least period $n$,
for every positive integer $n$. Let $(X_B,\sigma_B)$ be the SFT
induced by the graph $\mathcal{G}_B$. $(X_B,\sigma_B)$ is flow
equivalent to $(X_A,\sigma_A)$. Since $y_1,\ldots,y_n,y'_1,\ldots,y'_n$
are fixed points in $(X_B,\sigma_B)$ and $\sigma_B$ is mixing with
points of all least periods, there is an inert automorphism
$u\in\text{Aut}(\sigma_B)$ such that $u(y_i)=y'_i$ for all
$i=1, 2, \ldots, n$ \cite{BoF1}. Extend $u$ to a flow equivalence
$\widehat{u}: \mapt X_B\to \mapt X_B$ by $\widehat{u}([x,t])=[u(x),t]$. Let
$G:\mapt X_A\to \mapt X_B$ be a flow equivalence arising from the
construction. Then $F=G^{-1}\widehat{u}G$ is the required flow
equivalence, i.e., $F(\mathcal{C}_i)=\mathcal{C}'_i$ for all
$i=1, 2, \ldots, n$.
\end{proof}
In contrast to Theorem \ref{theorem:234},
 note that if a flow equivalence $F$ maps a cross section
$C$ onto a cross section $D$, then the return maps to these cross sections
are topologically conjugate. The action of $\mathcal F_A$
on cross sections is very far from transitive.

The center of the automorphism group of an irreducible shift of
finite type is simply the powers of the shift \cite{Ryan}. The next
result
(from \cite{Cthesis2011}) is
the analogue for the mapping class group.

\begin{theorem}\label{theorem:235}
Suppose $(X_A,\sigma_A)$ is a nontrivial irreducible shift of finite
type.
Then the center of $\mathcal{M}_A$ is trivial.
\begin{proof}
Let $\mathcal{C}$ be a circle in $\mapt X_A$ and $F$ be an element in
the center of $\mathcal{M}_A$. Suppose that $F(\mathcal{C})\neq
\mathcal{C}$. Note that $F(\mathcal{C})$ is also a circle. Then
there is a flow equivalence $G$ such that
$G(\mathcal{C})=\mathcal{C}$ and $G(F(\mathcal{C}))\neq
F(\mathcal{C})$ by Theorem~\ref{theorem:234}. Thus
$FG(\mathcal{C})=F(\mathcal{C})\neq GF(\mathcal{C})$ which is a
contradiction. Hence $F(\mathcal{C})=\mathcal{C}$ for all circles
$\mathcal{C}$ in $\mapt X_A$. Therefore, $F$ is isotopic to the identity
by Theorem~\ref{theorem:231}.
\end{proof}
\end{theorem}

\begin{remark}\label{theoremtoquestion}Suppose $\sigma_A$ and $\sigma_B$ are nontrivial irreducible
  SFTs. It is not known whether
  $\text{Aut}(\sigma_A)$ must embed as a subgroup of $\text{Aut}(\sigma_B)$.
  Kim and Roush proved the embedding does exist when $\sigma_A$ is a
  full shift \cite{KR1990}. With mapping class groups in place of automorphism
  groups, we do not have even the analogue of the Kim-Roush
  result. (Adapting the automorphism group argument of Kim and Roush to mapping
  class groups, using flow codes in place of block codes, is
  problematic.)
\end{remark}

  \begin{question}
Do all mapping class groups of nontrivial
irreducible SFTs embed into each other?
\end{question}

Recall
$\mathcal M^o_A$ denotes  the kernel of the Bowen-Franks
representation $\mathcal M_A \to \cok (I-A)$.
We are led to several questions about
$\mathcal M^o_A$.
\begin{question} Let $\sigma_A$ be a  nontrivial irreducible SFT.
  Is $\mathcal M^o_A$  simple? perfect (i.e. equal to its commutator
  subgroup)? generated by involutions?
\end{question}
There has recently been a burst of results constraining the
  structure of an
  automorphism group of a subshift (usually assumed to be minimal)
of low complexity (e.g. polynomial complexity,
  or even just zero entropy). (See
\cite{ ST2015, CQY, CK2016, DDMP2016,CFKP}
and their references.) Here degree $d$ polynomial complexity of
  a subshift means that
the number of allowed words of length $n$ is bounded by a polynomial
$p(n)$ of degree $d$. The classes of zero entropy shifts, degree $d$
polynomial complexity shifts and minimal shifts are each invariant
under flow equivalence.

\begin{question} Are there constraints on the structure of the mapping
class group of a low complexity (minimal) shift, analogous to
constraints on the automorphism group?
\end{question}

Some quite interesting full groups have been proved to be
finitely generated or even finitely presented
\cite{JM,Matui}.

\begin{question}
 Let $\sigma_A$ be a  nontrivial irreducible SFT.
Is $\mathcal M^o_A$
finitely generated?
\end{question}

Because $\rho_A $ is surjective, and the group of automorphisms of a
finitely generated abelian group is itself finitely generated, we have
that
 $\mathcal M_A$  is finitely generated if
 $\mathcal M^o_A$  is finitely generated.
(In contrast, the group of automorphisms of the dimension module
of $X_A$ is often but not always a finitely generated group
\cite{BLR}.)

\section{Outer and nonspatial automorphisms} \label{sec:autos}

In this section we show that $\mathcal M_A$
has an outer automorphism.
Extending work from \cite{mbwk:amsess}, 
we give examples of
 $\text{Aut}(\sigma_A)$
with outer and nonspatial automorphisms,  
and derive consequences of  spatiality of isomorphisms from
sufficiently rich subgroups of 
$\text{Aut}(\sigma_A)$.

It is natural to suspect that nontrivial irreducible SFTs
$\sigma_A$, $\sigma_B$ which are not flow
equivalent cannot have isomorphic mapping class
groups. (Although, given works of Riordam, Matsumoto and Matui
(see \cite{Rordam,MM}), one could speculate that
isomorphism of their Bowen-Franks groups alone might imply
$\mathcal M_A \cong \mathcal M_B$.)
Question \ref{feinduced} gives
one standard approach to this possibility.

\begin{definition} An isomorphism $\phi : G_1\to G_2$ between
  groups of homeomorphisms is {\it spatial} if it is induced
  by some homeomorphism $H$ (i.e., $\phi (g) = H^{-1}gH$).
\end{definition}

\begin{question} \label{feinduced} Suppose $(X_A,\sigma_A)$ and $(X_B, \sigma_B)$  are
  nontrivial irreducible shifts of finite type.
  Is every isomorphism
$\psi\colon \mathcal M^o_A \to \mathcal M^o_B$ spatial?
  Is every isomorphism
$\psi\colon \mathcal M_A \to \mathcal M_B$ spatial?
\end{question}

\begin{remark} \label{spatialremark}
  A standard  method for proving spatiality (and more)
  for a group $G$ of homeomorphisms  of the Cantor set
  (e.g. within full groups \cite{GPSfull,GM,JM})
  appeals to $G$ having a sufficiently rich supply of
  maps which are the identity on large open sets.
  It's problematic (perhaps impossible) to find some
  analogue of this
  approach for $\mathcal M_A$ or $\text{Aut}(\sigma_A)$.
  The only  element of $\text{Aut}(\sigma_A)$
  which is the identity on a nonempty open set is the
  identity; and if $F$ in $\mathcal F_A$ is the identity
  on an open neighborhood of a cross section, then $[F]$ is the identity
  in $\mathcal M_A$.
\end{remark}

Recall, a system $(X,T)$ is
{\it indecomposable} if $X$ is not the union of
  two disjoint nonempty $T$-invariant subsystems. Equivalently,
  $\mapt X$ is connected. 

\begin{definition}
For an indecomposable system $(X,T)$, the  extended
mapping class group  of
$T$, $\mathcal M^{\text{ext}} (T)$,
is the
group $\mathcal H(T)$ of all homeomorphisms $\mapt X\to \mapt X$,
modulo the subgroup of those isotopic to the identity
in $\mathcal H(T)$.
\end{definition}

With $\mapt X$ connected, an element
of $\mathcal H(T)$ either respects orientation on all orbits
or reverses orientation on all orbits.
The mapping torus of $(X,T^{-1})$ can be identified
with the mapping torus of $(X,T)$,
but with its unit speed suspension flow moving
in the opposite direction. With this identification,
$\mathcal M(T)=\mathcal M(T^{-1})$.
An orientation reversing homeomorphism $V$ of $\mapt X$
is a flow equivalence from $T$ to $T^{-1}$.
Such a $V$ always exists when $\sigma_A$ is a nontrivial 
irreducible SFT, because $(\sigma_A)^{-1}$ is conjugate to the
SFT presented by the transpose of $A$,  and the complete
invariants agree on $A$ and its transpose.
Clearly $\mathcal M(T) $ is an index 2
normal subgroup of $\mathcal M^{\text{ext}}(T)$. 

\begin{theorem} \label{outertheorem}
  Let $\sigma_A$ be a nontrivial irreducible SFT. 
  The action  of the extended
  mapping class group $\mathcal M^{\text{ext}} (\sigma_A)$
  by permutations on the circles in $\mapt X_A$ is faithful.
  Consequently, the center of $\mathcal M^{\text{ext}}(\sigma_A)$
  is trivial and the outer automorphism group of
  $\mathcal M (\sigma_A)$ has cardinality at least two.
\end{theorem}
\begin{proof}
  Suppose $F$ and $G$ are homeomorphisms of $\mapt X_A$,
  with the same action by permutations on circles.
  If $FG^{-1}$ is orientation preserving, then $FG^{-1}$ is
  isotopic to the identity,
  by Corollary \ref{theorem:232}, so $[F]=[G]$ in
  $\mathcal M^{\text{ext}}_A$. 
  Now suppose $F$ is orientation preserving and $G$ is orientation
  reversing. For definiteness, after passing to isotopic maps,
  we suppose they are given by flow codes. Let  $W,V$ be distinct words
  such that $(WV^nW)$ is an $X_A$-word, for all $n$.
  With $O=WV^n$, consider a circle $C$ which is the suspension of
  a periodic orbit  for $\sigma_A$ with defining block
  $(O)V^N(OO)V^N(OOO)V^N$. For $n$ sufficiently large, and then $N$
  sufficiently larger than $n$,
  there will be large integers $M,P$ and words
  $\overline O, \overline V, \widetilde O, \widetilde V$ 
  with $\overline V^M  $ much longer than $\overline O\overline O\overline O$
  and $\widetilde V^P$ much longer than
  $\widetilde O\widetilde O\widetilde O$, such that 
  the circles
  $FC$ and $GC$ will be suspensions of $\sigma_A$-orbits
  with defining blocks of the following forms:
  \begin{align*}
\overline B\ =\ &    (\overline O)
        \overline V^M
    (\overline O\, \overline O)
    \overline V^M(\overline O\, \overline O\, \overline O)\overline V^M
    \ , \quad \text{for } FC \ , \\
\widetilde B \ =\ &       (\widetilde O)
    \widetilde V^P
    (\widetilde O\widetilde O\widetilde O)\widetilde V^P
    (\widetilde O\widetilde O)
    \widetilde V^P
    \ , \quad \quad \, \text{for } GC \ .
  \end{align*}
  Now the blocks interrupting  $\overline V$-periodicity in $(\overline B)^{\infty}$
  will have   $ \dots 123123123 \dots$
as  a periodic relative size pattern, 
  while the blocks 
  interrupting of $\widetilde V$-periodicity in $(\widetilde B)^{\infty}$
  will have 
  $\dots 321321321 \dots $
  as a periodic  
  relative size  pattern. Thus $F\mathcal C \neq G \mathcal C$.   
This finishes the proof of faithfulness.
  The proof of triviality of the center of  $\mathcal M^{\text{ext}}_A$
  follows the  proof of Theorem \ref{theorem:235}. Then,
  conjugation by an orientation reversing homeomorphism of
  $\mapt X_A$ defines an automorphism of $\mathcal M_A$
  which is not an inner automorphism of $\mathcal M_A$.
\end{proof}

We now turn to the automorphism group of $\sigma_A$. The
next definition formalizes a property used in 
\cite{mbwk:amsess}, as recalled below. 

\begin{definition}\label{defn:SIC}
  An SFT $\sigma_A$ is SIC 
if 
$\text{Aut}(\sigma_A)$ is the internal direct sum of
its center $\left< \sigma_A\right>$ and a complementary
normal subgroup containing the inert automorphism subgroup
$\text{Aut}_0(\sigma_A)$.
\end{definition}

We will show next that there are many examples of SIC SFTs.
We say $\lambda$ is rootless in $R$ if
$\lambda =u^k$ with $k\in \N$ and $u \in R$
implies $k=1, \lambda = u$.  For example,
a positive integer is rootless in $\Q$ if it is
rootless in $\Z$. A fundamental unit of a quadratic
number ring
$R$ is rootless in $R$. If $\lambda $ is an algebraic
number with infinite order, then it has a $k$th root 
 in $\Q(\lambda )$ for only finitely many $k$.  

\begin{proposition}
  Suppose $\sigma_A$ is a nontrivial irreducible SFT, and
  $\lambda_A$, the Perron eigenvalue of $A$, is rootless
  in $\Q (\lambda_A)$. Then $\sigma_A$ is SIC.
\end{proposition}
\begin{proof}
  One part of the dimension representation
   $\rho_A$ is the
  homomorphism $\mu$ which sends an automorphism $U$
  to the positive number by which $\rho_A (U)$ multiplies
  a Perron eigenvector of $A$.
  The image  group under multiplication,
  $\mu (\text{Aut}(\sigma_A)) := H$,
    is finitely generated free abelian, with 
  $\mu (\sigma_A)=\lambda_A$, the Perron eigenvalue of $A$.
  By the rootless assumption,
$H$ is
the internal direct sum of $ \left< \lambda_A \right>$
and some complementary group $N$.
The epimorphism $\text{Aut} (\sigma_A) \to H/N$ splits
(by $[\lambda_A^n] \mapsto \sigma_A^n$). 
Let $K = \mu^{-1} (N)$.
Because the complementary subgroup
$\left< \sigma_A\right>$ is the center,   
the group 
$\text{Aut} (\sigma_A)$ 
is the internal direct sum $K \oplus
  \left< \sigma_A \right>$.
\end{proof}


\begin{proposition}\label{SICproposition}
  Suppose $(X_A, \sigma_A)$ is a nontrivial SIC irreducible SFT,
  $\text{Aut}(\sigma_A) \cong K\oplus \left< \sigma_A\right>
  \cong K \oplus \Z$.
  Let $\phi $ be the automorphism of $\text{Aut}(\sigma_A) $ 
  which
  in the latter notation is 
  $(k,n) \mapsto (k,-n)$. 
  Then  $\phi$   is not spatial.
\end{proposition}
\begin{proof}
 Suppose $\phi$ is induced by a homeomorphism $H$. It follows that 
  $H$ is  a conjugacy from $\sigma_A$ to its inverse, 
  with $HU=UH$ for every $U$ in $K$. First suppose $\sigma_A$
  is mixing. Then for any periodic point $x$ of sufficiently large period,
  there is an inert automorphism $U$ such that $Ux=\sigma_Ax$.
  (This follows e.g. from any of the three papers
  \cite{BoF1, BLR, NasuSln}; for a precise argument, 
  see the   proof of Proposition   \ref{richsubgroups} below.)
  Thus $H$ commutes with $\sigma_A$ on
  a dense set, and hence everywhere. This contradicts
  $H^{-1}\sigma_AH = \sigma_A^{-1}$.

Now suppose  $\sigma_A$ is irreducible with period $p>1$.
  Then $\sigma_A$ 
    induces a cyclic permutation of $p$ disjoint
    clopen sets $B,  \sigma_A(B), \dots ,
    \sigma_A^{p-1}(B)$.
    After postcomposing $H$ with a power of $\sigma_A$, we may
    assume $H(B)=B$. 
    The return map $\sigma_A^p|_B$ is 
     a mixing
     SFT, and every inert automorphism of
     $\sigma_A^p|_B$ extends to an inert automorphism
    of $\sigma_A$. Thus  $H|_B$ commutes
    with $\sigma^p|_B$. Because $\sigma_A^p|_B$ has infinite order,
    this contradicts
      $H^{-1}\sigma_A^pH = \sigma_A^{-p}$.
\end{proof}

In \cite[Proposition 4.2]{mbwk:amsess},
the automorphism $\phi$ above was used to produce  
an example of a 
 nonspatial automorphism of $\text{Aut} (\sigma_A)$, 
for a mixing SFT $\sigma_A$
such that  $\text{Aut} (\sigma_A) \cong \text{Aut}_0 (\sigma_A)
\oplus \left<\sigma_A\right>$ and $\sigma_A$ is not conjugate
to its inverse. The proof in \cite{mbwk:amsess}
was simply to note that spatiality of $\phi$ would require
$\phi$ to be a (nonexistent) 
conjugacy from $\sigma_A$ to its 
inverse.

%

\begin{remark} For a nontrivial SIC mixing
  SFT $\sigma_A$ which is topologically conjugate to its inverse
  (such as a rootless full shift), the outer automorphism
  group of $\text{Aut}(\sigma_A)$ has cardinality at least four.
  (There is the nonspatial
  involution, and another element of order two in
  $\text{Out}(\sigma_A)$ arising from conjugating by a topological
  conjugacy of $\sigma_A$ and its inverse, essentially by the argument
  proving Theorem \ref{outertheorem}.)
  The action on periodic points
  of conjugacies of $\sigma_A$ and $\sigma_A^{-1}$
  is studied in \cite{mbwk:amsess,  KLP}.  
  \end{remark}

Although there can be nonspatial automorphisms of $\text{Aut}(\sigma_A)$,
we do not know whether this is possible for various distinguished subgroups
(such as the commutator). This motivates the following propositions.

\begin{proposition} \label{subspatial}
  Suppose $\sigma_A$ is a nontrivial irreducible SFT, $(X,T)$ is
  a zero dimensional system and 
  $H$ is a subgroup of $\text{Aut}(\sigma_A)$
  satisfying the following:
  \begin{enumerate}
  \item
    $\{ x\in X_A: \exists U \in H, Ux=\sigma_Ax\}$ is dense
    in $X_A$.
  \item
    The centralizer of $H$ in $\text{Aut}(\sigma_A)$
    equals $\left< \sigma_A \right>$. 
  \end{enumerate}
  Suppose $\phi: H \to \phi(H)$ is a spatial
  isomorphism to a subgroup
  of $\text{Aut}(T)$; i.e.,  $\phi: U \mapsto pU p^{-1}$,
  with $p: X_A \to X$ a homeomorphism. 
  Then $p^{-1}Tp $ equals $ \sigma_A$ or $ \sigma_A^{-1}$, and
  $p$ induces a spatial automorphism $\text{Aut}(\sigma_A) \to
  \text{Aut}(T)$. 
\end{proposition}
\begin{proof}
  Let $\psi = p^{-1}Tp $. By (1),
  $\psi \sigma_A = \sigma_A \psi$ on a dense set, hence
  everywhere. By (2), $\psi \in \left< \sigma_A \right>$.
  Because $\psi$ and $\sigma_A$ have equal entropy,
  $\psi$ equals $\sigma_A$ or $\sigma_A^{-1}$.
\end{proof}

\begin{proposition} \label{richsubgroups}
  Suppose $\sigma_A$ is a nontrivial mixing SFT, and 
  $H$ is a subgroup of
  $\text{Aut}(\sigma_A)$ containing the subgroup
  \[
  K:= \left<\{ aba^{-1}b^{-1}:
  \{a,b\} \subset \text{Aut}_0 (\sigma_A), a^2=b^2=Id\}\right> .
  \]
  Then $H$ satisfies the conditions (1) and (2) of
  Proposition \ref{subspatial}.
\end{proposition}
\begin{proof}
  Let $\mathcal P_n$ be be the set of  $\sigma_A$ orbits
  of cardinality $n$.
    Pick $N$ such that
    $n\geq N$ implies  $|\mathcal P_n| \geq 4$.
    Now suppose $n\geq N$.
    Given $x,y$ in distinct orbits in $\mathcal P_n$, 
    we can choose an inert involution $U(x,y)$ which exchanges
    $x$ and $y$ and is the identity on points of period at most
    $n$ which are not in the orbits of $x$ and $y$.
    (This follows from \cite[Lemma 2.3(a)]{BoF1}, and the freedom
    to ``vary the embedding'' stated in its proof.)
    Suppose $x,y,z$ are in distinct orbits in $\mathcal P_n$. 
    Let $a=U(x,y)$,  $b=U(y,z)$, $k(x,y,z) = aba^{-1}b^{-1} \in K$.
    Then $k(x,y,z)$ cyclically permutes $x,y,z$ and is the identity
    map on points of period at most $n$ outside the orbits of $a$, $b$
    and $c$. The map $k=k(\sigma_A ( x),y,z)k(x,y,z)k(x,y,z)$ satisfies
    $k(x)=\sigma_A(x)$; this shows $H$ satisfies (1). The maps
    $k(x,y,z)$ induce all 3-cycle permutations of 
    $\mathcal P_n$, and therefore $K$ induces all even permutations of
    $\mathcal P_n$. Because $|\mathcal P_n|\geq 4$, no nontrivial
    permutation of $\mathcal P_n$ commutes with every even permutation. 
     Thus  an automorphism in the centralizer of $K$ maps
  $\mathcal O$ to $\mathcal O$, 
  for all but finitely many of the
  finite orbits $\mathcal O$, 
 and thus must be a
  power of the shift. 
\end{proof} 

For mixing SFTs $\sigma_C$, let 
$G_C$ denote $\text{Aut}(\sigma_C)$ or  $\text{Aut}_0(\sigma_C)$,
and let $H_C$ denote some associated subgroup (such as the commutator,
or the subgroup generated by involutions) such that 
(i) $H_C$  satisfies the containment
assumption of Proposition \ref{richsubgroups}, 
and (ii) any group isomorphism $G_A \to G_B$ must restrict to an isomorphism
$H_A \to H_B$. Showing 
any isomorphism $H_A \to H_B$ must be spatial would show that
the group isomorphism class of $H_A$ (and also the group
isomorphism class of $G_A$) classifies $\sigma_A$ up to flip
conjugacy. 

\section{Invariant cross sections and automorphisms} \label{sec:sections}

In this section we show how some elements of the mapping class group
are
induced by automorphisms of flow equivalent systems, and show
for a nontrivial irreducible SFT $(X_A, \sigma_A)$
that these are (by far) not all of
$\mathcal M_A$. For
 $(X,T)$, let $\widetilde X$ denote the cross section
$\{ [x,0]\in \mapt X: x\in X\}$.


\begin{definition} If $u\in \text{Aut}(T)$, then $\widehat u
: \mapt X \to \mapt X$ is the flow equivalence  (actually a
self-conjugacy of the suspension flow) defined by
$\widehat u: [x,t)\mapsto [u(x),t)$, $0\leq t < 1$.
\end{definition}

\begin{definition}
Let $F:\mapt X\to \mapt X$ be a flow equivalence. A cross section $C$ of
$\mapt X$ is  an \textit{invariant cross section} for $F$ if
$F(C)=C$.
\end{definition}

For example, $\widetilde X$ is an invariant cross section
for $\widehat u$, for every $u$ in
$\text{Aut}(T)$.

\begin{definition}
An equivalence $F: \mapt X \to \mapt X$ is induced by an automorphism
$v$ of
the return map $\rho_C$ to an invariant cross section $C$  if
$F(y) = v(y) $ for all $y$ in $C$.
\end{definition}
If
flow equivalences $F,F'$ from $\mapt X$ to $\mapt X$
have the same invariant cross section
$C$, and $F(y)=F'(y)$ for all $y$ in $C$, then $F$ and $F'$ are
isotopic.

Now we can spell out a straightforward but useful correspondence.

\begin{theorem}\label{theorem:245}
  Let $T: X\to X$ be a homeomorphism of a compact zero
  dimensional metric space.

\begin{enumerate}
\item
Suppose $F:\mapt X\to \mapt X$ is a flow equivalence with an invariant
cross section $C$. Then $F$ is a flow equivalence
induced by an automorphism of the first return map $\rho_c$ under the
suspension flow.
\item
Conversely, suppose $(X',T')$ is another system, and
$\widetilde{X'} $ denotes the cross section
$\{ [x,0]: x\in X'\}$ of $\mapt X'$.
Suppose $F: \mapt X' \to \mapt X$ is a flow equivalence. Then
for every $u$ in $\text{Aut}(T')$,
$F\widehat u F^{-1}$ is a flow equivalence $\mapt X \to \mapt X$;
$F(\widetilde{X'})$ is an invariant cross section
for  $F\widehat u F^{-1}$; and
$(X',u)$ is  topologically conjugate to
the return map $F(\widetilde{X'})\to
F(\widetilde{X'})$ under
the suspension flow on $\mapt X$.
The map $u\mapsto F\widehat u F^{-1}$ induces a homomorphism
$\phi_F: \text{Aut}(T')  \to \mathcal M(T)$.
\end{enumerate}
\end{theorem}

\begin{proof}
For (1), let $u=F|_{C}$. Then $u:C\to C$ is a homeomorphism.
Therefore $u\in\text{Aut}(\rho_c)$.

For (2), the homomorphism $\phi_F$ is a composition of group
homomorphisms
\[
\text{Aut}(T') \to \mathcal F(T')\to
\mathcal F(T)\to \mathcal M(T)
\]
where $\mathcal F$ denotes the group of self flow equivalences.
The second homomorphism is bijective and
the third is  surjective.
%
\end{proof}

\begin{theorem}\label{theorem:241}
For a nontrivial irreducible SFT $(X_A, \sigma_A)$,
let $\phi$ be the map $\text{Aut}(\sigma_A) \to \mathcal M_A$ defined by
$u\mapsto \widehat u$. Then
$\text{Ker}(\phi)=\left<\sigma_A\right>$, the cyclic group generated
by $\sigma_A$.
\end{theorem}

\begin{proof}
Clearly
$\text{Ker}(\phi) \supset\left<\sigma_A\right>$. Now suppose
$u\in \text{Ker}(\phi)$. By Theorem \ref{theorem:231}, for every
circle $\mathcal C$ in $\mapt X_A$,
$\widehat{u}(\mathcal C) = \mathcal C$.
It follows that the automorphism $u$ maps each finite $\sigma_A$ orbit
to itself. Because $(X_A, \sigma_A)$ is an irreducible SFT,
it follows from  \cite[Theorem 2.5] {BK2},
that $u$ is a power of the shift.
\end{proof}

%

\begin{theorem} \label{theorem:242}
Suppose $(X_A, \sigma_A)$ is a nontrivial irreducible SFT.
 Then for every irreducible SFT $\sigma_B$ flow equivalent to
$\sigma_A$,
$\mathcal M_A$ contains a copy of $\text{Aut}(\sigma_B )/
\left<\sigma_B\right>$.
Every flow equivalence $\mapt X_A \to \mapt X_A$ with an invariant cross section
arises from an element
of some  such $\text{Aut}(\sigma_B )$, as described
in Theorem  \ref{theorem:245}.
\end{theorem}

\begin{proof} This follows from
Theorem \ref{theorem:245},
Theorem \ref{theorem:241}
and
the fact that
a homeomorphism flow equivalent
  to a nontrivial irreducible SFT must itself be a nontrivial irreducible
SFT.
\end{proof}

%

\begin{example}
We do not know if
there is any special algebraic
relationship  between the automorphism groups of flow equivalent
nontrivial irreducible SFTs (versus arbitrary nontrivial irreducible SFTs). We  show now that
if $(X_A,\sigma_A)$ and $(X_B,\sigma_B)$ are flow equivalent mixing SFTs, then
it is not necessarily true that the groups
$\text{Aut}(\sigma_A)/\left<\sigma_A\right>$ and
$\text{Aut}(\sigma_B)/\left<\sigma_B\right>$ are isomorphic.
Consider
\[
A=\left( \begin{array}{cc}
  1 & 1 \\
  1 & 0 \\
 \end{array}
 \right)\
 ,
 \quad \quad
 B=A^2
=
\left( \begin{array}{cc}
 2 & 1 \\
 1 & 1 \\
\end{array}
 \right)\
,
\quad  \quad
C=[2] \ .
\]

The matrices $B$ and $C$ define flow equivalent SFTs
(if $D$ is $B$ or $C$,
then $\text{coker}(I-D)$ is trivial and $\det{(I-D)}=-1$).
The center of the automorphism group of an irreducible SFT
is the powers of the shift \cite{Ryan}.
But in
$\text{Aut}(\sigma_B)$, the center has a square root
(because $\sigma_{A^2}$ is conjugate to $(\sigma_A)^2$), while
in $\text{Aut}(\sigma_C)$
 and the center does not, because the 2-shift does not
have a square root \cite{Lind84}.
\end{example}

\begin{proposition}\label{theorem:246}
Suppose $(X,T)$ contains a subsystem $(X',T')$
which is a nontrivial irreducible shift of finite type.
Suppose
$F\in\mathcal F(T)$ and $F$ maps
$\mapt X'$ (a subset of $\mapt X$)
 into itself but not onto itself.
Then no element of $[F]$ has an
invariant cross section.
%
\begin{proof}
  Any element of $[F]$ will also map $\mapt X'$ into itself but not onto itself.
  So it suffices to   suppose there is an invariant cross section
  $C$ for $F$,
  and derive a contradiction.
By
Proposition \ref{theorem:245},
$F:\mapt X\to \mapt X$ is induced by an automorphism $u$ of the
return map $\rho_c$ to
 $C$. The restriction
$\rho'$ of
$\rho_c$ to $C\cap \mapt X'$ is an irreducible SFT, because
it is flow equivalent to the irreducible SFT $(X',T')$, since
$C\cap \mapt X'$
is a cross section for the flow on $\mapt X'$. Therefore the restriction of
$u$ to $C\cap \mapt X'$, being an injection into
$C\cap \mapt X'$
commuting with $\rho'$,
must be a surjection. But this implies $F$
maps $\mapt X'$ onto itself, which is a contradiction.
\end{proof}
\end{proposition}

The next result, generalizing a construction from
\cite{Cthesis2011},
shows that flow equivalences  satisfying the assumptions of
Proposition \ref{theorem:246}  are abundant.
We don't understand much about them.

\begin{theorem}
  Let $(X_A,\sigma_A)$ be a nontrivial irreducible SFT.
Let  $(X',\sigma')$ be a proper subsystem which is a nontrivial
irreducible SFT.
Then there is an infinite collection of
flow equivalences $F:\mapt X_A\to \mapt X_A$, representing distinct elements
of $\mathcal M_A$,
such that $F$ maps $\mapt X'$ into itself but not onto itself
(and therefore no element of  $[F]$ has an invariant cross section).
\end{theorem}
\begin{proof}
  From the complete invariants
  for flow equivalence
of nontrivial irreducible SFTs, and Krieger's Embedding Theorem,
one can find a sequence  $X_1, X_2, \dots $ of  distinct
(even disjoint) nontrivial irreducible SFTs
which are proper subsystems of $X'$ and are
flow equivalent to $X'$.
   By the Extension Theorem in \cite{BCEbce},
 a flow
 equivalence $F'_n: \mapt X' \to \mapt X_n\subset \mapt X_A$
 extends to a flow equivalence
$F_n: \mapt X_A \to \mapt X_A$.
 The classes $[F_n]$ are distinct, because the images
 $F'_n(\mapt X')$ are distinct.
\end{proof}

Next we exhibit an example,
 not relying on an
 appeal to an extension theorem, of a flow equivalence $F$
 such that no element of 
$[F]$  has an 
 invariant cross section.

\begin{example}\label{concretefe}
  Let $\sigma : X\to X$ be the full shift on three symbols
  $\{0,1,2\}$. If $W=W_1W_2 ... $ is any sequence on these
  symbols and $W_1\neq 2$,   then $W$ has a unique prefix in the set
  $\mathcal W=\{ 00, 01, 02,1\}$;
  likewise, $W$ has a unique prefix in the set
  $\mathcal W'=\{ 10, 11, 12,0\}$.  Let $\mathcal W \to \mathcal W'$
  be the bijection  given by
  $00 \mapsto 0, 01 \mapsto 10, 02 \mapsto 12, 1 \mapsto 11$.
  We claim there is a flow equivalence
  $F: \mapt X\to \mapt X$
  corresponding to the change $2W \to 2W'$ wherever
  $W\in \mathcal W$ and $2W$ occurs in a point of $X$.
  Let $X' \subset X$ be the full 2-shift on symbols
  $\{1,2\}$; let $X''$ be the points of
  $X'$ in which the word $212$ does not occur. Then $F$ maps
  $\mapt X'$ onto $\mapt X''$, a proper subset of
  $\mapt X'$, so no element of $[F]$ has an invariant cross section. 

  To be precise, we will construct $F$ as a flow code, as described
  in the appendix. First, we define a discrete cross section
  $C$ of $X$ as the disjoint union of two ``state sets'' $V_0$ and $V_1$,
  with
  $V_0=\{ x \in X: x_{-1}=2\}$,
  $V_1=\{x \in X: x_{-2}x_{-1} \in
  \{ 21, 00, 01, 10,11\}$. If $x\in C$, and $k$ is the least positive
  integer such that $\sigma^k(x) \in C$, then
  $x_0\dots x_{k-1}$
  is a $C$-return word $W$, of length $k$ (here $k$ is 0 or 1).
  Whether  $\sigma^k(x)$ is in $V_0$ or $V_1$ is determined by the state
  set containing $x$ and the return word $W$.
  Thus the return words can be used to label edges of a directed
  graph with states $V_0, V_1$. The adjacency matrix $\widetilde A$
  of this word-labeled graph (whose entries are formal sums of labeling words),
  and the adjacency matrix $A$ of the
  underlying graph, are as follows:
  \[
  \widetilde A =
  \begin{pmatrix} 2+ 02 &00 + 01 +1  \\
    2  &  0+1
  \end{pmatrix} \ ,
  \qquad
  A =
  \begin{pmatrix} 2 & 3 \\
    1  & 2
  \end{pmatrix} \ .
  \]
  Similarly, we define another discrete cross section, $C'$,
  as the disjoint union of state sets
$V_0=\{ x \in X: x_{-1}=2\}$,
  $V'_1=\{x \in X: x_{-2}x_{-1} \in
  \{ 20, 00, 01, 10,11\}$. As happened with $C$,
  the $C'$ return words label edges of a graph with states
  $V_0$ and $V'_1$, with labeled and unlabeled adjacency
  matrices
\[
  \widetilde{A'} =
  \begin{pmatrix} 2+ 12 & 0 + 10+ 11  \\
    2  &  0+1
  \end{pmatrix} \ ,
  \qquad
  A'=A =
  \begin{pmatrix} 2 & 3 \\
    1  & 2
  \end{pmatrix} \ .
  \]
  Now we may define a homeomorphism $\phi:C\to C'$,
  taking $V_0$ to $V_0$ and $V_1 $ to $V'_1$, by a
  $C, C'$ word block code $W_0 \mapsto W'_0$
  described by an input-output
  automaton which simply changes word labels:
  \[
  \begin{pmatrix}
    2 \ \to  \ 2 \ ,\
    02  \ \to  \ 12 \ \quad \ \ &
    00 \ \to  \   0\ , \
    01  \ \to  \ 10\ ,\
    1 \ \to  \ 11  \\
    2 \ \to  \ 2  &  0 \ \to  \ 0\ ,  1 \ \to  \ 1
  \end{pmatrix} \ .
  \]
  This $\phi$ is a conjugacy of the return maps to $C$ and $C'$
  (each of which is conjugate to the SFT $\sigma_A$). The
  induced map $\mapt \phi: \mapt  X \to \mapt X$
  is the  flow equivalence $F$ we require.
  \end{example}

\begin{question}
Is the mapping class group of a nontrivial irreducible SFT
generated by elements which
 have an invariant cross section?
\end{question}

\begin{proposition}\label{theorem:249}
Let $(X_A,\sigma_A)$ be a nontrivial irreducible SFT, with 
$F\in\mathcal F_A$. If there is a circle $\mathcal{C}$ such that
$\{F^n(\mathcal{C}):n\in\mathbb{N}\}$ is an infinite collection of
circles then no element of $[F]$ has an invariant cross section.
\begin{proof}
  If $F$ has an invariant cross section $C$, then $F$ is determined
  up to isotopy by 
  an automorphism $U$ of the return map $\rho_C$.
  As $\rho_C$ is another
  irreducible SFT, every periodic point of $\rho_C$ lies in a finite
  $U$-invariant set, so every circle in $X_A$ lies in a finite
  $F$-invariant set of circles.  
\end{proof}
\end{proposition}
We do not know if the converse to Proposition \ref{theorem:249}
  is true.
  \begin{example}
    In Example \ref{concretefe}, the forward
    $F$ orbit of the circle through the periodic
    orbit $(21)^{\infty}$ is the union of  infinitely many circles
    (those through the periodic orbits of $(21^n)^{\infty}$, $n \geq 1$).
    \end{example}

\section{Residual finiteness} \label{sec:residual}
\begin{definition}
  Let $G$ be a group. $G$ is {\it residually finite} if for every pair of
distinct  elements $g,h$ in $G$, there is a homomorphism $\phi$ from
$G$ to a finite group such that $\phi (g) \neq \phi(h)$.
\end{definition}

 The automorphism group of a
subshift need not be residually finite. There is a  minimal subshift whose
  automorphism
group contains a copy of $\Q$ \cite{BLR}, and therefore is not
residually finite.  At another extreme, we thank V. Salo for pointing out to us
residual finiteness often fails to hold for reducible systems,
as in
work in progress of Salo and Schraudner, and examples such as the
following, related to examples in \cite{SaloGroups2015}.
Let $S_{\infty}$ denote the increasing union of the groups $S_n$,
the permutations of $\{ 1, 2, \dots , n\}$,
identified with the permutations $\pi$ of $\N$ such that
 $\pi(k)=k$ if $k>n$. 
Then $S_{\infty}$ contains
$A_{\infty}$, the increasing union of the alternating groups $A_n$.
Because $A_{\infty}$ is an infinite simple group, it is not residually
finite. Let
$A=\left( \begin{smallmatrix}
    1&1&0\\0&1&1\\0&0&1\end{smallmatrix}\right)$. One
easily checks that
$\text{Aut}(\sigma_A)$ contains a copy of $S_{\infty}$, and thus is
not residually finite.

In contrast, the automorphism group of an irreducible shift of finite type
(or any subshift with dense periodic points) is
  residually finite \cite{BLR}.

\begin{theorem} \label{notresiduallyfinitetheorem}
  Let $X_A$ be a nontrivial irreducible SFT. Then
  $\mathcal{M}_A$ is not residually finite.
\begin{proof}
For a proof, it
suffices to define a monomorphism $S_{\infty} \to \mathcal M_A$.
After passing from $X_A$ to a topologically conjugate shift,
we may assume that there is a symbol $\alpha$ such that
there are infinitely many distinct words $V_1, V_2, \dots $ such that
for all $k$,
$\alpha V_k \alpha$ is an allowed word and $\alpha$ does not
occur in $V_k$. Informally, an element $\pi$ of $S_{\infty}$ will
act simply by replacing words $\alpha V_k \alpha$ with
$\alpha V_{\pi(k)} \alpha$.

To make this precise we use flow codes (described in Appendix
\ref{sec:flowcodes}).
For $n$ in $\N$, define $\ell (n) =|V_n| +1$, and
$K_n =\{x\in X_A: x_0\dots x_{\ell(n)} = \alpha V_n \alpha \}$.
Given $N$, define a discrete cross section
\[
C_N = X_A \setminus
\Big( \cup_{n=1}^N \cup_{j=1}^{\ell (n) -1} \sigma_A^j K_n \Big)\ .
\]
Let $\mathcal W_N$ be the set of return words to $C_N$. 
Given $\pi $ in $S_N$, define a word block code
\begin{align*}
  \Phi_{\pi} :\  \mathcal W_N &\to \mathcal W_N \\
    \alpha V_j  &\mapsto  \alpha V_{\pi (j)} \ , \quad 1\leq j \leq N \ ,\\
   W &\mapsto W \ , \quad \quad \ \ 
  \text{if }W\text{ is a symbol .}
\end{align*}
$\Phi_{\pi}$ defines a continuous map 
$\phi_{\pi}: C_N \to C_N$.
The rule $\pi \mapsto \phi_{\pi}$ defines a  monomorphism from $S_N$
into the group of homeomorphisms $C_N\to C_N$,
and therefore
$\pi \mapsto \mapt \phi_{\pi}$ defines a group monomorphism
$S_N \to \mathcal F_A$. It is then easy to see (from distinct actions
on periodic orbits) that $\pi \mapsto [\phi_{\pi}]$ is a group
monomorphism $S_N\to \mathcal M_A$.  Finally, the definition
of $\phi_{\pi}$ does not change with increasing $N$,
so we have an embedding $S_{\infty} \to \mathcal F_A$ producing 
 the embedding $S_{\infty} \to \mathcal M_A$. 
\end{proof}
\end{theorem}



The sofic groups introduced by Gromov
are an important
 simultaneous generalization of amenable and
residually finite groups.
(See e.g. \cite{CL2015,Pestov2008, Weiss2000} for definitions and  a start on the
large literature around sofic groups)
So far, no countable group has been proven to be
nonsofic. The mapping class group of a nontrivial irreducible SFT
$\sigma_A$ is
not residually finite, and it is not amenable (as $\mathcal M_A$
contains
a copy of $\text{Aut}(\sigma_A)/\left<\sigma_A\right>$, which contains free
groups
\cite{BLR}).

\begin{question}
Is $\mathcal M^o_A $ a sofic group?
\end{question}

\begin{remark}
With a somewhat more complicated proof appealing to canonical covers,
we expect that the basic idea of
Theorem \ref{notresiduallyfinitetheorem} can be used to show that the
mapping class group of a positive entropy sofic shift
 is not
residually finite.
Likewise, we expect  a subshift which is a  positive entropy
synchronized system \cite{BH1986}
will have a mapping class group which is not residually finite.
\end{remark}

\section{Solvable word problem} \label{sec:solvable}
The purpose of this section is to prove
Theorem \ref{mcgdecide}, which shows
that  the mapping class group
of an irreducible SFT has solvable word problem.
We begin with definitions and context.

The {\it alphabet} $\mathcal A(T)$ of a subshift $(X,T)$ is its symbol set.
For $j\leq k$, $\mathcal W(X,j,k) $ denotes
$\{x_j \dots x_k: x\in X\}$, the words of length $k-j+1$
occurring in points of $X$.
The language of a subshift $(X,T)$ is $\cup_{n\geq 0} \mathcal W(X,0,n)$.

\begin{definition} A subshift $(X,T)$ has a decidable language
  if there is an algorithm which given any finite word $W$ on
$\mathcal A (T)$ decides whether $W$ is in the language of $X$.
\end{definition}

\begin{definition} A group $G$ has solvable word problem if for
  every finite subset $E$  of $G$ there is an algorithm which given any
  product $g=g_m \dots g_1$ of elements of $E$ decides whether
  $g$ is the identity.
  \end{definition}

An old observation of Kitchens \cite{BLR} notes
that   the automorphism group of a shift of finite type has a solvable word
problem.
We thank Mike Hochman for communicating to us
the following sharper result.

\begin{proposition} \label{kitch}
  Suppose $(X,T) $ is a subshift with decidable language
  (for example, any shift of finite type).
  Then  $\text{Aut}(T)$ has
  solvable word problem.
\end{proposition}
\begin{proof}
Given $E=\{ \phi_1, \dots , \phi_m\} \subset \text{Aut}(T)$,
there are $N\in \N$ and functions
$\Phi_i: \mathcal W(-N,N)(T)  \to \mathcal{A} (T)$,
  $1\leq i \leq m$, such that
  $\Phi_i$ defines $\phi_i$ as a block code, i.e.
  for all $x$ and $n$,
  $(\phi_i x)_n = \Phi_i (x_{n-N} \dots x_{n+N})$.
  Suppose $k\in \N$ and $\phi=\phi_{j_k}\dots \phi_{j_1}$. Then
  for all $x $ in $X$, $(\phi x)_i =
  \Phi (x_{i-kN} \dots   x_{i+kN})$, where $\Phi$
  is a rule  mechanically
  computed   from the rules $\Phi_k , \dots , \Phi_1$
  \cite{H}.
However,  the domain of $\Phi$ might properly contain the
set $\mathcal W(-kN,kN)$ (even when the set $\mathcal W(-N,N)$
used to define the $\Phi_i$ is known).
  The map $\phi$ is the identity if and only if
  $    \Phi (x_{-kN} \dots   x_{kN}) = x_0$ for all
  words $x_{-kN} \dots   x_{kN}$  in
  $\mathcal W(-kN,kN)$; because $(X,T)$ has
decidable language,
  this set is known and can be checked.
\end{proof}

\begin{definition} \label{explicit}
A locally constant  function $p$ on $X$
is {\it given by an explicit rule} if for some $N$
there is given a
function $P$ from some superset of
$ \mathcal W(X,-N,N)$  to $\Z$ such that
for all $x$ in $X$, $p(x) = P(x_{-N} \dots x_N)$
  (or if $p$ is given by data from which such a $P$
could be algorithmically produced).
\end{definition}

\begin{definition}
  A subshift $(X,T)$ has solvable $\Z$-cocycle triviality
  problem if there is an algorithm which decides for any
  explicitly given  continuous (i.e. locally constant)
  function $p: X\to \Z$
  whether there is a continuous function $q: X\to Z$ such
  that $p=(q\circ T) - q$ (i.e., $p$ is a coboundary
  in $C(X,\Z)$, with transfer
  function $q$).
\end{definition}

If a subshift $(X,T)$ has solvable word problem,
then for an explicitly given $p$ in $C(X,\Z)$  known to
be a coboundary there is a procedure which will
produce an explicitly defined  $q$ such that
$p=(q\circ T) - q$ (enumerate the possible $q$
and test them).

For a positive integer $j$,
   a subshift $(X,T)$ with language $\mathcal L$
    is a {\it j-step} shift of finite
    type if for all words $U,V,W$ in $\mathcal L$, if $V$ has
    length $j$ and $UV\in \mathcal L$ and $VW \in \mathcal L$,
    then $UVW\in \mathcal L$. 
  
\begin{remark}
  As is well known, for an  irreducible  
$j$-step shift of finite type
$(X,T)$, and $p$ defined by $P,N$ as in Definition \ref{explicit},
the following are equivalent. 
\begin{enumerate}
\item  There is a continuous $q: X\to \R$ such that
  $p=(q\circ T) - q$.
\item
   There is a continuous $q: X\to \Z $ such that
   $p=(q\circ T) - q$, and for all $x$, $q(x)$ depends only
   on the word $x_{-N}\dots x_N$.
 \item If $x\in X$ and
   $T^k(x)=x$
and  $k\leq M:=\max \{j+1, 2N+1\}$, then
   $\sum_{n=0}^{k-1} p(T^nx) = 0 $.
   \end{enumerate}
(Here, $(1) \implies (2)$ because $T$ has a dense orbit. 
After passage to the $M$-block presentation
  (see e.g. \cite[Prop. 1.5.12]{LM}),
  $p$ presents as an edge labeling on an edge SFT, and 
the implication
$(3)\implies (1)$ reduces to an old graph argument
(recalled in \cite[Lemma 6.1]{BCEfei}), which
also gives
a decent algorithm for producing the
transfer function $q$ of (2).)
\end{remark}
Clearly, an irreducible SFT
has solvable
$\Z $-cocycle triviality problem.

To prove Theorem \ref{mcgdecide}, we emulate the proof of
Proposition \ref{kitch}, using flow codes in place of
block codes.
There are two difficulties.
First, we need for flow codes a computational analogue of
composition of block codes.
This is addressed in Appendix \ref{sec:flowcodes}.
Second, we need an algorithm
to determine
triviality of $[F]$ in $\mathcal M(T)$
when $F$ is given by a flow code.
We  address the latter issue now.

A subshift $(X,T)$ is  infinite if the set $X$ contains
infinitely many points. A subshift is transitive if
it has a dense orbit.

\begin{lemma} \label{discretetrviality}
  Suppose $(X,T) $ is a subshift,
 $C$ is an explicitly given
    discrete   cross section  for $(X,T)$
    and $\phi : C\to D $
    is a flow code defined
    by an explicitly given word code $(\Phi , C)$.

    Then the following are equivalent.
    \begin{enumerate}
    \item $[\mapt \phi]$ is trivial in $\mathcal M(T)$.
    \item
      There is a continuous function  $b : C \to \Z$
      such that  for all
  $x$ in $C$, the following hold:
      \begin{enumerate}
    \item
    The word $W'_0(x)$ equals the word
    $x_{b (x)} \dots x_{b (x) + |W'_0|-1} $ .
  \item
    $b (x) + |W'_0(x)| = |W_0(x)| + b (T^{|W_0(x)|}(x) ) $
  \end{enumerate}
  \end{enumerate}
   \end{lemma}

\begin{proof}
    Let $\alpha_t$ denote the time $t$ map of the suspension flow
    on $\mapt X$.
     Let $\Phi : W_{-N} \dots W_N \to W'$ be
     the explicitly given word code for $\phi$,
     mapping $(2N+1)$-blocks of
  $C$-return words
  to a return word for $D$. For $x$ in $C$
  with return block $W_{-N}^N(x)$, there is a
  concrete description of return times of $x$ to $C$ and
  $\mapt \phi (x)$ to $D$:
  \begin{equation} \label{returns}
    \tau_C(x)= |W_0(x)|\quad\quad
    \text{and}\quad\quad
    \tau_D (\mapt \phi (x)) = |W'_0(x)| \ .
  \end{equation}
  The condition (2)(b) states that the functions
  $x\mapsto \tau_C(x)$ and $x\mapsto \tau_D(\mapt \phi (x))$ are
  cohomologous in
$C(C,\Z)$, with respect to the return map
 $\rho_C: x\mapsto T^{|W_0(x)|}(x)$. 

For  a flow equivalence
$F : \mapt X\to \mapt X$ which maps each orbit to itself,
and maps a cross section $C$ onto a cross section $D$,
the following conditions  are equivalent
 (see e.g. \cite[Theorem 3.1]{BCEfei}):
  \begin{enumerate}
  \item $F$  is trivial in $\mathcal (T)$.
  \item
There is a continuous function $\beta : \mapt X \to \R$
 such that  $F: y \mapsto
  \alpha_{\beta (y)} (y)$, for all $y$ in $\mapt X$.
\end{enumerate}
  In the case $F= \mapt \phi$, given the second condition,
  $\beta$ must assume integer values on $C$.
  Define $b : C \to \Z$ by
  $b: x\mapsto \beta ([x,0])$.
  Then $b$ satisfies the conditions (a),(b) of the Lemma.
  The (cocycle) condition (b) is a consistency condition:
  each side of (b) gives a computation of the
  number $t$  (strictly greater than $b(x)$)
  such that $\mapt \phi $ maps the orbit interval
  $\{ \alpha_s ([x,0]): 0\leq s \leq |W_0(x)|\}$
  onto the orbit interval
  $\{ \alpha_s ([x,0]): b(x)\leq s \leq t\}$.

  Conversely, suppose $b: C\to \Z$ is a continuous function
 satisfying (a) and (b).
  By induction, using the given
  word block code, we see that for all $x$ in $C$ and
  all nonnegative integers $k$, for $s= \sum_{i=0}^k |W'_i(x)| $
 we have
 $(\phi x)_0 \dots (\phi x)_k =
 x_{b(x)} \dots x_{s-1}$. Because the return map to $C$ is a
 homeomorphism, we then have for $r = \sum_{i=-k}^{-1} |W'_i(x)|$
 that
 $(\phi x)_{-r} \dots (\phi x)_{-1} =
x_{b(x)-r}\dots  x_{b(x)-1} $.
Consequently, for all $x$ in $C$,
$\mapt \phi : [x,0] \to [T^{b(x)}, 0]$,
    so $\mapt \phi $ maps each flow orbit to itself.
  Finally, from $b$ we can define the continuous function
  $\beta$ of condition (2), as follows.
  For $x$ in $C$, set $r(x) = |W'_0(x)| / |W_0(x)|$ and
  \[
  \beta :  [x,t] \mapsto [x, b(x)+ tr(x)] \ , \quad
  x\in C, 0\leq t < |W_0(x)| \ .
  \]
  This rule defines $\beta$ on the entire mapping torus.
  The piecewise linearity
  of $\beta$  on the flow segments between returns to the cross section
  agrees with the flow code definition.
\end{proof}

\begin{lemma}   \label{mcgdecidelemma}
  Suppose $(X,T) $ is a transitive  subshift
  (for example, any irreducible SFT) with
decidable language and  solvable $\Z$-cocycle
triviality problem.
Suppose $C$ is an explicitly given
    discrete   cross section  for $(X,T)$
    and $\phi : C\to \overline{C} $
    is a flow code defined
    by an explicitly given word code $(\Phi , C)$.

  Then there is a procedure which decides whether
    $\mapt \phi$ is a flow equivalence  $\mapt X \to \mapt X$
such that $[\mapt \phi]$ is trivial in $\mathcal M(T)$.
   \end{lemma}

\begin{proof}
  We will decide whether there is a function $b \in C(X, \Z) $
  satisfying the conditions (a),(b) of
  Lemma \ref{discretetrviality}. We are explicitly given
  the locally constant return time functions
  $\tau_C(x) = |W_0(x)|$ and $\tau_D(\phi x) = |W'_0(\phi x)|$ .
  Because there is a dense $T$ orbit, a solution $b$ to
  (b) is unique up to an additive constant. Thus,
either
  every solution to (b) also satisfies (a), or no
  solution to (b) also satisfies (a).

  By the $\Z$-cocyle triviality and solvable word
  problem assumptions,
  there is an algorithm
  which produces  $b\in C(X,Z)$ such that
  $|W'_0(x)| - |W_0(x)| = b (T^{|W_0(x)|}(x) ) -b(x)  $,
  if such a $b$ exists. So,
  suppose  we have $b\in C(X,\Z)$ satisfying condition
  (b) of Lemma  \ref{discretetrviality}.
  From the explicitly given $C$ and word code $\phi$,
  we can compute a positive integer $M$ such that
  for all $x$, the word $x_{-M} \dots x_M$ determines
  $\mathcal W_{-N}^N(x)$ (and thus $W'_0(x)$),
  and also $M > \max |b| + \max |W'_0|$. Now whether
  (a) holds can be detected by testing each word
  $x_{-M}\dots x_M$ occuring in $X$. By the assumed decidability
  of the language, there is an algorithm
  to list these words.
  \end{proof}

\begin{theorem} \label{mcgdecide}
  Suppose $(X,T) $ is a transitive  subshift
 with
decidable language and  solvable $\Z$-cocycle
triviality problem   (for example, any irreducible SFT).  Then
 $\mathcal M(T)$ has solvable word problem.
   \end{theorem}

\begin{proof}
Given  $[F_1], \dots , [F_k]$ in
 $\mathcal M(T)$,   
for $1\leq i \leq k$ let
$\Phi_i$ be an explicitly given word code defining
a homeomorphism of discrete cross sections, 
$\phi_i :C_i \to D_i$, with $\mapt \phi_i$ isotopic to $F_i$.
Suppose $[F]= [F_{i_j}][F_{i_{j-1}}]\cdots [F_{i_1}] $ in
$\mathcal M (T)$.

By Proposition \ref{propcomp}, there is an algorithm
which computes a
rule
$\Phi $,  defining a homeomorphism $\phi: C \to D$ of
explicitly given cross sections of $(X,T)$,
such that $[\mapt \phi]=[F]$.
By Lemma \ref{mcgdecidelemma},
there is then a procedure which decides whether
    $[\mapt \phi]$
is trivial in $\mathcal M(T)$.
\end{proof}

\section{Conjugacy classes of involutions}
\label{Conjugacyclassesofinvolutions}

 Throughout this section, $A$ is a matrix defining a nontrivial
  irreducible SFT.
We will prove and exploit Theorem \ref{conjugacyclass},
which shows how conjugacy classes of
many involutions
in $\mathcal M_A$ are classified as $G$-flow equivalence classes of
mixing $G$-SFTs, for $G= \Z_2 :=\Z /2\Z$.

We prepare for the statement of
Theorem \ref{conjugacyclass}
with some definitions and background.
In this paper, by a $G$-SFT we mean
 a shift of finite type  together with a continuous
(not necessarily free) action of a finite group $G$
 by homeomorphisms
which commute  with the shift.
 A $G$-SFT is mixing (irreducible) if it is
mixing (irreducible) as an SFT.
A continuous $G$ action on an SFT $X_A$ lifts to a
continuous $G$
action on its mapping torus $\mapt X_A$.
Two $G$-SFTs are $G$-flow equivalent if there
is an orientation preserving homeomorphism between their mapping tori
which intertwines the induced $G$ actions.

 Recall, if $C$ is a cross section for a
 flow equivalence $F: \mapt X_A \to \mapt X_A$,
 and $\rho_C: C \to C$ is the return map to $C$
 under the flow, then $\rho_C$ is flow equivalent to
 $\sigma_A$ and in particular is a nontrivial irreducible SFT.
 If $C$ is also invariant under an involution $V$ in $\mathcal F_A$,
 then the pair $T=(\rho_C, V|_C)$ is a $\Z_2$-SFT; we say
 this $\Z_2$-SFT is {\it associated}
 to $V$, and to $\mapt X_A$.

\begin{theorem} \label{conjugacyclass}
     For $i=1,2$:
     suppose  $V_i$  in $\mathcal F_A$ is an involution.
     Then $V_i$ has an invariant cross section.
     Let  $C_i$ be any invariant cross section for $V_i$,
     with $T_i$  the associated $\Z_2$-SFT. Then the following are
     equivalent.
     \begin{enumerate}
       \item   $[V_1]$ and $[V_2]$ are conjugate elements in
  $\mathcal M_A$.
       \item
  The $\Z_2$-SFTs $T_1$ and $T_2$ are
  $\Z_2$-flow equivalent.
\end{enumerate}
\end{theorem}
\begin{proof}
  The involutions $V_1, V_2$ have invariant cross sections
  by Lemma \ref{icforinvolution}. By Lemma
  \ref{conjugacy of involutions is conjugacy}, there is
  an involution $V$ in $\mathcal F_A$
  which equals $V_2$ on $C_2$ (and therefore
  defines the same associated $\Z_2$-SFT), such that there is a
  flow equivalence $J$ such that $J^{-1}V_1J=V$. This shows
  the two $\Z_2$-SFTs are $\Z_2$-flow equivalent.
\end{proof}

If $V$ is an  involution  in $\mathcal F_A$, 
then the fixed point set of its restriction to an invariant cross section $C$ 
will, as  a subsystem of $(C, \rho_C)$, be an SFT.  
Theorem \ref{conjugacyclass} shows that the flow equivalence class
of this SFT  is an invariant of the conjugacy class of $[V]$
in $\mathcal M_A$, even though there can
be other elements $W$ in $[V]$ (but not other
involutions) with fixed point set containing a submapping torus
whose intersection with $C$ properly contains 
$C\cap \text{Fix}(V)$ and represents a different
flow equivalence class. 

\begin{question} \label{involutions}
  Suppose $[F]$ is an involution in $\mathcal M_A$.
  Is there an involution $V$ such that $[F]=[V]$?
\end{question}

If the answer to Question \ref{involutions} is yes,
then Theorem \ref{conjugacyclass} applies to all order
two elements of the mapping class group; if the answer
is no, then the quotient map $\mathcal F_A \to \mathcal M_A$
does not split.

  Below, for visual simplicity, where a point $x$ in $X_A$
 denotes a point in $\mapt X_A$,
 it denotes $[x,0]$. We similarly abuse notation for sets.
 
\begin{lemma} \label{icforinvolution}
  Suppose $V \in \mathcal F_A$ is an involution.
  Then
  $V$ has an invariant cross section.
\end{lemma}
\begin{proof}
      Suppose  $X_A\cap V(X_A) $ is nonempty
  (if it is empty,
  then $X_A\cup V(X_A)$   is an invariant cross section for $V$).
Fix $\epsilon >0 $  small enough that the image under  $V$
    of any orbit interval of length $2\epsilon$ has length less than
    1.
For 
     a clopen subset $C$ of $X_A$ 
     containing  $X_A\cap V(X_A)$, with 
     $V(C) \subset X_A \times (-\epsilon, \epsilon )$,
     define $C'$ to be the clopen-in-$X_A$ set of points $x'$ such
     that for some $t$ in $(-\epsilon , \epsilon )$ and
     some $x$ in $C$, $V(x) = [x',t]$. Fix $C$ small enough that we
     also have 
     $V(C') \subset X_A \times (-\epsilon, \epsilon )$, and set 
    $D= C\cup C'$. Now there is a continuous 
 involution
    $h: D \to D$ with $h(C)=C'$, and a continuous 
  function 
  $\gamma :  D\times (-\epsilon, \epsilon)\to \R$, such that
    for all  $[x,t]$ in $D\times (-\epsilon, \epsilon)$,
    \[
    V: [x,t] \to [h(x), t+\gamma ([x,t])]\ . 
    \] 
For every $x$ in $D$, $V$ maps the interval
$\{ [x,t]: - \epsilon < t < \epsilon  \}$ by an orientation
preserving homeomorphism to  some orbit interval of length
less than 1. In particular, if $h(x)=x$, then
$\gamma (x)=0$ (otherwise, $V$ would map the orbit
segment between $x$ and $Vx$ onto itself reversing
endpoints, and thus reversing orientation).
Define 
\begin{align*}
  K& = \{  x\in D:  \gamma (x) \geq 0 \}
  \cup
  \{ [h(x),\gamma (h(x))]: x\in D, \gamma (x) \geq 0\}\ ,
  \\
  L&=  X_A \setminus D \ , \\
  E&=K \cup L \cup VL \ . 
\end{align*}
We will show $E$ is an invariant
cross section for $V$. Invariance is clear, since
for $x$ in $D$, we have 
$V(x)=  [h(x), \gamma (h(x))]$.
  
  Suppose $x\in D$. 
  Let $K(x) = K\cap (\{x\} \times (-\epsilon, \epsilon ))$; 
  then $K = \cup_{x\in D} K(x)$. Let $y=h(x)$.
We have $K(x) \subset \{ x, [x, \gamma (y)]\}$.  
Either both  $\gamma (y)$ and $\gamma (x)$ are zero,
or they are nonzero with opposite sign. Thus
\begin{align*}
  K(x) &= \{ x, [x,\gamma(y)]\} =
    \{ x \} \ , \quad \text{if } \gamma (x) = 0 \ , \\
    &= \{ x  \}  \ ,
    \quad  \qquad \qquad \quad \quad \  \ 
    \text{if } \gamma (x) > 0 \ , \\
    &= \{ [x, \gamma (y)]  \}  \ ,
      \quad  \qquad \ \ \ \ \ \ 
      \text{if } \gamma (x) <  0 \ .
\end{align*} 
For $x$ in $D$, define 
$\kappa (x)= \max \{ \gamma (x), \gamma (h(x))\}$.
It follows that 
$K  = \{ [x, \kappa (x)]: x\in D \} $, the graph of a continuous
function on $D$. The sets
$K,L,VL$ are disjoint. It is now straightforward to verify
that  $E$ is closed,
$E$ intersects every flow orbit and the return time
function on $E$ is continuous. Thus $E$ is a cross section. 

%
%
%
%
%
\end{proof}

Below, by the
normalized suspension flow over a cross section $C$,
we mean
the suspension flow after a time change such that points
move at unit speed  and points in $C$
have return time 1. This can be achieved by a flow equivalence
from the mapping torus of  the return map $\rho_C$.

\begin{lemma} \label{conjugacy of involutions is conjugacy}
  Suppose $U$ and $W$ are involutions  $\mapt X_A \to \mapt X_A$,
  with $[U]$ and $[W]$ conjugate in $\mathcal M_A$.
 Then $[W]$ contains
 an involution $V$ such that the following hold.
 \begin{enumerate}
   \item
     There is an invariant cross section $C$ for $W$
and for $V$ such that $V=W$ on  $C$.
\item
$V$  is an isomorphism of a normalized
  suspension
  flow over an invariant cross section.
\item
  There
  is a flow equivalence $J: \mapt X_A \to \mapt X_A$ such that
  $J^{-1}UJ = V$.
  \end{enumerate}
\end{lemma}
\begin{proof}
  By Lemma \ref{icforinvolution}, $W$ has an invariant cross section; 
  without loss of generality, we  assume 
it is $X_A$.
  Given $x$ in $X$, let $\gamma_x : [0,1] \to [0,1]$ be the homeomorphism
  such that $W: [x,t] \mapsto [W(x), \gamma_x(t)]$, $0\leq t \leq 1$.
  Then define $B: \mapt X_A \to \mapt X_A $ to
  be  $[x,t]\mapsto [x,\gamma_x^{-1}(t)]$, $0\leq t \leq 1$,
  and set $V=UB$. $V$ is the required map.

  By assumption there is $F$ such that
  $[F]^{-1}[U][F]=[V]$, so
  $F^{-1}UF = VH$, where
  there is a continuous function $\beta : \mapt X_A \to \R$
  such that   $H$ is a flow equivalence
  defined by $H(x) = \alpha_{\beta (x)}(x)$
(for this see e.g. \cite[Theorem 3.1]{BCEfei}).
  It suffices to
  find a flow equivalence $G$ such that $G(VH)G^{-1}= V$.
  Define functions  $b = \max \{ \beta ,  0\} $ and
  $c = \min \{ \beta , 0\}$. Then
define  $H_+$ and $H_-$ on $\mapt X_A$ by
$    H_+(x) = \alpha_{b (x)} (x)$ and
$    H_-(x) = \alpha_{c(x)}(x) $.

  The set in an orbit on which $\beta$ is nonzero is a disjoint
  union of intervals;  on each,  $\beta $
  has constant sign, and on each, $H$ is a surjective
  self-homeomorphism respecting the flow orientation.
  Now by continuity of the functions $b$ and $c$,
  $H_+$ and $H_+$ are flow equivalences of $\mapt X_A$,
  isotopic to the identity. Clearly $H=H_+H_- = H_-H_+$.

  We have $V =HVH $, because
$  V = V(VHVH) = V^2(HVH)$.
  Next, we show that $H_-VH_+ = V$. For all $x$,
  \[
  VH(x) = V(\alpha_{\beta(x)}(x))  =
  \alpha_{\beta(x)}(Vx) := z \ .
  \]
  Then $H(VH(x))= \alpha_{\beta(z)} (z)$, and
  \begin{align*}
    x &= VH(VH(x)) = V ( \alpha_{\beta(z)} (z) )
    =  \alpha_{\beta(z)} (Vz)
    =  \alpha_{\beta(z)} V \alpha_{\beta(x)} (Vx) \\
      &  =  \alpha_{\beta(z)}  \alpha_{\beta(x)} (V^2x)
        =  \alpha_{\beta(z)+\beta(x)}  (x) \ .
  \end{align*}
  Thus $\beta(z)+\beta(x) =0$ on the dense set of aperiodic
  points, hence everywhere. Because the sign of $\beta (z)$
  is the same as the sign of $\beta $ on
  $H(z)  = HVH (x) = V(x)$, it follows that
  $\beta $ is nonzero at $x$ if and only if $\beta$  is nonzero
  with opposite sign   at $V(x)$. So, if $\beta (x)>0$,
  then $H_-VH_+ (x) = HVH (x) =V(x)$. If $\beta (y) < 0$,
  then $y=Vx$ with $\beta (Vx) >0$, so
  \[
  H_-VH_+ (y)= H_-VH_+ (Vx) = V(Vx) = x = V(y) \ .
  \]
  Finally, let $G= H_-$. Then
  \[
  G(VH)G^{-1} =   H_-(VH_+H_-)(H_-)^{-1} = H_-VH_+ =V \ .
  \]  \end{proof}

  We give more information now on the $G$-SFTs.

  A free $G$-SFT is a $G$-SFT for which the $G$-action is free.
  By a construction of Parry explained in \cite{BS}
  (also see \cite[Appendix A]{BoSc2}),
free $G$-SFTs can be presented by square matrices with entries in
$\Z_+G$, the set of elements $\sum_g n_g g$ in the integral group ring
$\Z G$ with every $n_g$ a nonnegative integer.
Let $\text{El}(n,\Z G)$ be the group of $n\times n$ elementary matrices
over the integral group ring $\Z G$.

\begin{theorem}\cite{BS}  \label{classifstatement}
  Suppose $A,B$ are square  matrices over $\Z_+G$,
  presenting  nontrivial irreducible free 
$G$-SFTs. 
Then the following are equivalent.
\begin{enumerate}
\item
  The  $G$-SFTs are $G$-flow equivalent.
\item
 $I-A$ and $I-B$ are stably elementary
  equivalent over $\Z G$, in the following sense:
  there exist integers $j,k,n$ and identity matrices
$I_j, I_k$ with sizes $j,k$
  such that there are matrices $U,V$ in
$\text{El}(n,\Z G)$
  such that $U((I-A)\oplus I_j)V = (I-B) \oplus I_k$.
\end{enumerate}
\end{theorem} 

The classification statement Theorem \ref{classifstatement}
is the main special case
of
\cite[Theorem 6.4]{BS},  given one additional
simplifying remark. When a $\Z_+G$ matrix $A$ defines
a nontrivial irreducible $G$-SFT $T$,
the matrix $A$
must be essentially irreducible (so,
\cite[Theorem 6.4]{BS}
applies) 
 and its ``weights group'' in
the statement of \cite[Theorem 6.4]{BS}
must be all of $G$.  (Otherwise, as an SFT $T$ could
not have a dense orbit; see \cite[Prop.D.7]{BCEgfe} for
detail.)

There is also a complete (more complicated)
classification of $G$-flow equivalence
for general  free $G$-SFTs, in  \cite{BCEgfe}.
In the nonfree case, significant invariants
are known, but the  classification problem is open.
Still, we will see  with the remainder of the
section that tools for $G$-SFTs are of some use for
learning about conjugacy classes in $\mathcal M_A$. 
Define
  \[
  \mathcal C_A = \{ [V]\in \mathcal M_A:
  V^2 = Id \text{ and } V
  \text{ is associated to  a free } \Z_2-SFT \}\ .
  \]
  It follows from
  Proposition \ref{conjugacy of involutions is conjugacy} that
  \[
  \mathcal C_A = \{ [V]\in \mathcal M_A:
   V 
  \text{ is a fixed point  free involution}\} \ . 
  \]
  We will say an $n\times n$ matrix $D$ over $\Z$ is a Smith normal form
  if $D$ is a
  diagonal matrix $\text{diag}(d_1 , d_2 , \dots  , d_n )$
  satisfying the following conditions:
  $d_{i+1}$ divides $d_i$ whenever $1 \leq i < n$ and $d_{i+1} \neq 0$;
  $d_{i+1} = 0$ implies $d_i =0$; and $d_i \geq 0$
  if $i > 1$. 
  It is well known that any $n \times n$ matrix $B$ over $\Z$ is
  $\text{SL}(n, \Z)$  equivalent (hence $\text{El}(n,\Z)$ equivalent) to
  a unique Smith normal form, which we denote $\text{Sm}(B)$.
  (Our \lq\lq Smith normal form\rq\rq
  is slightly unconventional,
 following \cite{BS},
  to address sign and achieve
  $\text{Sm}(B\oplus I_k) = \text{Sm}(B) \oplus I_k$.)
    Note, $\det (B) = \det ( \text{Sm}(B) )$.
  \begin{theorem} Suppose $\mathcal \sigma_A$ is a nontrivial
  irreducible SFT and 
 $\text{det}(I-A)$
  is an odd integer. Then $\mathcal C_A$ is the union of
  finitely many conjugacy classes in $\mathcal M_A$.  
\end{theorem}
\begin{proof}
  Let $C$ be a matrix  over $\Z_+\Z_2$ presenting
  a free   $\Z_2$-SFT which is $\Z_2$-flow equivalent
  to a $\Z_2$-SFT associated to a free involution in $\mathcal F_A$. 
  Let  $C= eX + g Y$,  with
  $X$ and $Y$ over $\Z_+$ and $G=\{ e,g\}$. 
  The matrix $F=\left( \begin{smallmatrix} X&Y \\ Y&X
  \end{smallmatrix} \right)$ defines an SFT flow
  equivalent to $\sigma_A$, so  
   $ \det (I-A) = \det (I-F) $, and therefore 
  \begin{equation} \label{deteq} 
  \det(I-A) = \det (I-(X+Y)) \det (I-(X-Y)) \ .
  \end{equation}
  In our  special situation, with 
  $G=\Z_2$ and $\det (I-F) $ is odd,
  by   \cite[Theorem 8.1]{BS} 
  the stable $\text{El}(\Z G)$ equivalence class of $C$ (and thus the
  $G$-flow equivalence class of the $G$-SFT it presents)
  is determined by the pair $\big(\text{Sm}(I-(X+Y)),\text{Sm}(I-(X-Y)\big)$.
  Because $\det (I-A) \neq 0$, it follows from 
\eqref{deteq} that only finitely many
  pairs are possible.
\end{proof}

Let us say a $\Z_2$-SFT is {\it inert} if the involution defining
the $\Z_2$ action is an inert automorphism of the underlying
SFT. This is the class of greatest interest to us. (For
example, these involutions induce involutions on the mapping
torus which are in the kernel of the Bowen-Franks representation.)

  \begin{theorem} Suppose $\mathcal \sigma_A$ is a nontrivial
    mixing SFT and there is an inert $\Z_2$-SFT associated to a
    free involution $V$ in $\mathcal F_A$.  Then $\text{det}(I-A)$
  is an odd integer. 
\end{theorem}
  \begin{proof}
    Suppose  $(\sigma_B, U)$ is an SFT with free inert involution $U$
    associated to $V$.
      Ulf Fiebig proved that a finite order 
  automorphism $U$ of $\sigma_A$ is inert if and only if the homeomorphism
  induced by the shift on  the quotient space of $U$-orbits
   has the same zeta function as $\sigma_A$
  \cite[Theorem B]{Fiebig93}. 
  This condition forces $\det (I-tB) =1 \mod 2$,  by the
  argument of \cite[Lemma 2.2]{KRzp}, 
    so $\det (I-B)$ is odd. 
    Because $\sigma_A$ and $\sigma_B $ are flow equivalent,
    $\det (I-B)=\det(I-A)$.
  \end{proof}

  \begin{remark} 
  If $\det(I-A)$ is an odd negative squarefree integer, then
  $\sigma_A$ is flow equivalent to a full shift with a free
  inert involution, and there is a free inert $\Z_2$-SFT associated
  to an involution of $\mapt X_A$.
  We expect it is possible to prove  such involutions
  exist whenever $\det (I-A)$ is odd,
  by direct construction or by appealing to
%
%
the following difficult result  of   
Kim and Roush.  
\begin{theorem} \cite[Theorem 7.2 and Lemma 2.2]{KRzp} 
  Let $\sigma_A$ be a mixing shift of finite type and let $p$
  be a prime. 
  Then the following are equivalent.
  \begin{enumerate}
  \item
    There is  an   SFT $\sigma_B$ shift equivalent to $\sigma_A$,
    and an  order $p$ automorphism $U$ of $\sigma_B$,
    and a factor map $\pi: X_A \to X_B$ 
    for which the fiber over every point is a cardinality $p$
    orbit of $U$. 
  \item
For all positive integers $n$,
    \[
    o_n \ \geq \ \frac{p-1}p o_{n/p} +
    \frac{p-1}{p^2} o_{n/p^2} +
    \frac{p-1}{p^3} o_{n/p^3} +  \cdots
    \]
    where $o_k$ denotes the number of $\sigma_A$ orbits of
    cardinality $k$.
    \end{enumerate} 
  \end{theorem}
The condition (2) above implies $\det(I-tA)=1$
    mod $p$. Condition (2)  holds for all $n$ if it holds up to
    a computable bound. (See \cite[Sections 1-2]{KRzp} for more 
    explanation.) 
    The automorphism $U$ in (1) must be inert
(by  \cite[Theorem B]{Fiebig93}), so there will
  be an inert $\Z_p$-SFT associated to a free $\Z_p$ action on 
  $\mapt X_B$.   The shifts 
  $\sigma_A$ and  $\sigma_B$ in (1) are flow equivalent,
  so there will also be 
an inert $\Z_p$-SFT associated to  a free $\Z_p$ action on 
  $\mapt X_A$.
\end{remark} 

  While the classification of non-free $\Z_2$ SFTs is open,
we can  consider the invariant which is the
flow equivalence class of the fixed point shift. The
following result of Long \cite[Theorem 1.1]{Long}
is an instrument for creating examples.

\begin{theorem} (Long) \label{long} 
  Let $f$ be an inert automorphism of a mixing shift of finite type
  $X$ with $\text{Fix}_f(X)\subset Y$ 
  where $Y \neq X$ and $Y$ is a $f$-invariant subshift of finite type
  in $X$. Suppose $n \geq 2$ and $n$ is the smallest
  positive integer such that $f^n = Id$.
  If the restriction of $f$ to $Y$ is inert, then there exists an inert
  automorphism of $X$, $U$, such that $Y$ is the fixed point shift of $U$,
  where $U^n = id$ and $n$ is the minimal
  positive integer $k$ such that $U^k = id$.
\end{theorem}

For example, let $f$ be the inert involution of the full
shift on symbols $0,1,2$ which exchanges the symbols $0$ and $1$.
For  a positive integer $n$, let
$T_{n}$ be the subshift with language $(\{ 0,1\}^n2)^*$
(words of length $n$ on $\{ 0,1 \}$ alternate with the symbol $2$).
Then $T_{n}$ is invariant under $f$, and
one can check the restriction of $f$
to $T_{n}$ is inert. By Long's theorem, $T_{n}$ is  the
fixed point shift of some inert involution of the 3-shift.
$T_n$ is an irreducible SFT with Bowen-Franks group
$\Z /(2^n-1)\Z$.  So, infinitely many flow equivalence classes
occur as the fixed shift of an inert $\Z_2-SFT$ associated to an
involution of $\mapt X_A$, and those involutions must
represent distinct elements of $\mathcal M_A$.

One can more generally
produce infinitely many distinct flow equivalence classes
of inert involutions of $\Z_2$-SFTs associated to $\mapt X_A$,
whenever there is a free $\Z_2$-SFT associated to $\mapt X_A$, 
by combining  some of Long's
results (\cite[Theorem 1.1, Theorem 1.2, Lemma B.2]{Long}) 
and some construction work (e.g., for $k$ in $\N$ embed into $X_A$
$2k$ disjoint copies of an SFT admitting an inert involution,
say using \cite{BoF1}).

%

\appendix

\section{Flow codes} \label{sec:flowcodes}

Flow codes were developed in \cite{BCEfei} as a flow map
analogue of block codes. In \cite{BCEfei},
flow codes were considered for not necessarily invertible
flow maps. In this appendix, for simplicity we only
consider flow equivalences, and ``flow code'' means
``flow code'' for a flow equivalence.


First we recall some definitions from \cite{BCEfei}.
Let $C$ be a discrete cross section for a subshift $X$.
Given $C$, the {\it return time bisequence} of a point $x$ in $C$
is the bisequence $(r_n)_{n\in \Z}$
(with $r_n=r_n(x)$) such that
\begin{enumerate}
\item
 $\sigma^j (x)\in C$ if and only if
$j=r_n$ for some $n$,
\item
 $r_n< r_{n+1}$ for all $n$, and
\item
$r_0= 0$.
\end{enumerate}
A {\it return word} is a word equal to $x[0,r_1(x))$ for some $x\in
C$.
Given $x\in C$ and $n\in \Z$, $W_n=W_n(x)$ denotes the return
word $x[r_n,r_{n+1})$.
In the context of a given $C$, when we write
$x=\dots W_{-1}W_0W_1 \dots $ below, we mean $x\in C$ and $W_n = W_n(x)$.
Given $x\in C$ and $i\leq j$, the tuple
$(W_n(x))_{n=i}^j$ is the $[i,j]$ return block of $x$,
also denoted $W_i^j(x)$, and
$\mathcal W(i,j,C) = \{W_i^j(x) : x\in C\} $.
To know this return block is to know
the word $W= W_i\cdots W_j$ together with
its factorization as a concatenation of  return words.

\begin{definition}\label{defnwordblockcode}
  Suppose $C$ is a  discrete cross sections
  of a subshift $(X,T)$.
A $C$  {\it word block code}
is a function $\Phi :\mathcal W(-N,N,C) \to \mathcal W'_0$,
where $\mathcal W'_0$ is a set of words
and $N$ is a nonnegative integer.
A  {\it word block code}
is a $C$  {\it word block code} for some $C$.
The  function $\phi$ from
$C$ into a subshift
 given by  $\Phi$
is defined to map
$x = (W_n)_{n\in \Z} $
to the concatenation
$x' =   (W'_n)_{n\in \Z} $, with   $W'_n=  \Phi (W_{n-M},...,W_{n+M}) $
and $x'[0,\infty )=W'_0W'_1 \dots \ $ .
\end{definition}

For $D$ a discrete cross section of a subshift $(X',T')$, a
$C,D$ flow code  is a $C$ word block code $\Phi$ defined
as above, with the following additional properties:
\begin{enumerate}
\item
  $\mathcal W'_0$ is the set of $D$ return words
\item
  The induced map $\phi$ is a homeomorphism  $\phi: C \to D$
  which is a topological conjugacy of the  return maps
  of $C$ and $D$ (with respect to $T$ and  $T'$).
\end{enumerate}

In this case we refer to  $(\Phi , C,D) $ as a flow code
defining $\phi$. This code induces a flow equivalence
$\mapt \phi : \mapt X \to \mapt X'$  by the following rule,
in which $r(x) = |W'_0(x)|/ |W_0(x)|$:
\begin{align*}
\mapt \phi :
      [x,t] &\mapsto [\phi (x), t r(x)]  \ , \quad
      \text{ if }x\in C \text{ and }
        0\leq t < |W_0(x)|\ .
\end{align*}
By \cite[Theorem 5.1]{BCEfei},
for every flow equivalence of subshifts, $F: \mapt X \to \mapt X'$,
there is a flow code $(\Phi ,C,D)$ inducing $ \mapt \phi :
\mapt X \to \mapt X'$ such that $\mapt \phi$  is isotopic to
$F$.

Now suppose  flow codes $(\Phi_1,C_1, D_1)$ and
$(\Phi_2,C_2, D_2)$ induce $\phi_1: C_1 \to D_1$ and
$\phi_2: C_2 \to D_2$, with
$\mapt \phi_1: \mapt X_1 \to \mapt X_2 $ and
$\mapt \phi_2: \mapt X_2 \to \mapt X_3 $.
If $D_1=C_2$, then we can compose the word block codes as easily
as we compose  block codes to obtain a flow code $(\Phi , C_1, D_2)$
defining $\phi : C_1 \to D_2$ such that
$\mapt \phi = (\mapt \phi_2)\circ (\mapt \phi_1)$.
But if $C_1 \neq D_1$, we need another ingredient to
produce a flow code
 as a function of the given data  such that
$[\mapt \phi] = [(\mapt \phi_2)\circ (\mapt \phi_1)]$.

  Given a discrete cross section $C$  for $(X,T)$
 and $x\in X$,
  define $\tau (x,C) = \min \{i\geq 0: T^i(x) \in C\}$.
  The next proposition adapts the Parry-Sullivan  argument
  (\cite{PS}; see Theorem \ref{pstheorem})
to discrete cross sections. As the argument is very similar to
arguments given in \cite{BCEfei}, we will leave
a proof as an exercise (perhaps after reviewing \cite{BCEfei}).

  \begin{proposition}\label{pslemma}
    Suppose $C$ and $D$ are discrete cross sections for a subshift
 $(X,T)$. Define
    \begin{align*}
      K = K&(C,D) =
(C\cap D) \, \bigcup  \,
      \{x\in C\setminus D: \tau (x,D ) < \tau (x,C) \} \subset C\\
      L = L&(C,D) = \{ \sigma^{\tau (x,D)}(x):  x\in K \} \subset D \\
      \delta  :  K &\to L \ , \quad \delta :
      x \mapsto T^{\tau (x,D)}
      \ \ .
    \end{align*}
    Then
    \begin{enumerate}
\item
$K$ and $L$ are discrete cross sections for $(X,T)$.
\item
  $\delta $ is a well defined homeomorphism.
\item
  $\delta $ is a topological conjugacy of the return maps to $K$ and $L$
  under $T$, i.e.
$\delta \rho_K
  = \rho_L \delta $.
\item $\delta$ is given by a word block code
  $\Delta: \mathcal W(K,0,1)\to \mathcal W(L)$. \\
    $\delta^{-1}$ is given by a word block code
  $\Psi : \mathcal W(L,-1,0)\to \mathcal  W(K)$.
\item
  $[ \mapt \delta ]  $ is the identity element in $\mathcal M (T)$.
\item
  \label{algitem}
  Suppose the subshift $(X,T)$ has decidable language.
Let there be given $N$ in $\N$ and
word sets $\mathcal V_C$, $\mathcal V_D$ such that
$C = \{ x\in X: x_{-N} \dots x_N \in \mathcal V_C\}$ and
$D = \{ x\in X: x_{-N} \dots x_N \in \mathcal V_D\}$. Then
there is an algorithm to determine $M$ in $\N$ and
the following:
\begin{enumerate}
\item  $\mathcal V_K\subset \mathcal W(X) $ such that
  $K=\{ x\in X: x_{-N} \dots x_{M+N} \in \mathcal V_K\}$.
\item
    $\mathcal V_L\subset \mathcal W(X) $ such that
  $L=\{ x\in X: x_{-(M+N)} \dots x_{M+N} \in \mathcal V_L\}$.
 \item The word codes $\Delta $ and $\Psi$.
\end{enumerate}
  \end{enumerate}
\end{proposition}

  In Part (6) above, the decidability of the language lets us
  find an upper bound  to the return time to $K$.

  We say a  discrete cross section $C$ for a subshift $(X,T)$ is
  {\it explicitly given} if there is  given $N$ in $\N$
  and a subset $\mathcal V_C$ of the language of $X$
  such that $C = \{ x\in X: x[-N,N] \in \mathcal V\}$
  (or if $C$ is given by data from which such a set $\mathcal V$
  could be algorithmically produced).
  Similarly, a flow code $(\Phi ,C,D)$ is explicitly given
  if $C$ is explicitly given and for some $M$,
  $\Phi$ is given as
  a function from a subset of $\mathcal W(C, -M,M)$ 
  (or by algorithmically equivalent information).

  \begin{proposition}\label{propcomp}
    For $i=1\dots ,k+1$, let
    $(X_i,T_i)$ be a subshift
 with decidable language.
 Suppose for $1\leq i \leq k$ that
    $\mapt \phi_i: \mapt X_i \to \mapt X_{i+1}$ is a flow equivalence
    defined from a homeomorphism $\phi_i: C_i \to D_i $
    defined by an explicitly given  flow code $(\Phi_i, C_i, D_i)$.

    Then there is an algorithm which produces
    an explicitly given flow code $(\Phi , E, \overline{E})$,
    with $E\subset C_1$ and $\overline E\subset D_k$,
    inducing $\phi: E_1 \to E_k$
such that
    $(\mapt \phi_k)\circ (\mapt \phi_1)$ and $\mapt \phi $  are
isotopic.
        \end{proposition}

  Note, we are not claiming to produce a  $\phi $
  such that
  $(\mapt \phi_2)\circ (\mapt \phi_1) = \mapt \phi $.

  \begin{proof}[Proof of Proposition \ref{propcomp}]
    By induction, it suffices to prove the proposition
    assuming $k=2$.

    For the explicitly given discrete cross sections $D_1$ and
    $C_2$ of $(X_2, T_2)$, we have explicitly  from Lemma \ref{pslemma}
    $K=K(D_1, C_2)$,
    $L=L(D_1,C_2)$ and a
    flow code $(\Delta ,K,L)$ for $\delta$.
     Define  $E = \phi_1^{-1} (K)$ and
    $\overline{E} = \phi_2(L)$.
    Set $\psi_1 = \phi_1 |_E$,  $\psi_2 = \phi_2|L$ and
    $\phi = \psi_2 \delta \psi_1$.
    Now $\mapt \phi_i$ and $\mapt \psi_i$ are isotopic, for  $i=1,2$,
    and
  $\mapt \delta$ is isotopic to
    the identity, by Lemma \ref{pslemma}.
Therefore
    $(\mapt \phi_2)\circ (\mapt \phi_1)$ is isotopic to
    $\mapt \phi $.

From the explicitly given word codes for $\phi_1$ and $\phi_2$,
we can compute explicitly
    a flow code $(\Psi_1 , E, K) $ for $\psi_1$ and
    a flow code $(\Psi_2 , L, \overline E) $ for $\psi_2$.
    Now the discrete cross sections align,
    and we can compose the word codes
     $(\Psi_2 , L, \overline E) $, $(\Delta ,K,L)$,
$(\Psi_1 , E, K) $ to obtain a
    block word code
    rule $(\Phi , E, \overline E)$ for $\phi: E \to \overline E$,
with $\Phi$ defined for some $M$ on a set $\mathcal W$
    containing $\{ \mathcal W_{-M}^M(x): x\in E\} $.
    (Moreover,   by solvability of the word problem for $(X_1,T_1)$,  we
    may  then choose to shrink
    $\mathcal W$ so that the containment becomes equality.)
\end{proof}


\bibliographystyle{plain}
\bibliography{mcg}

\end{document}